\documentclass[11pt]{amsart}
\usepackage[greek,english]{babel}
\languageattribute{greek}{polutoniko}
\usepackage[utf8x]{inputenc}
\usepackage[T1]{fontenc}
\usepackage{a4wide}
\usepackage{mathpazo}
\usepackage[dvipsnames]{xcolor}
\usepackage[colorlinks=true,allcolors=MidnightBlue,linktoc=page]{hyperref}

\usepackage{mathrsfs}
\usepackage{amssymb, amsmath}
\usepackage{enumitem}
\usepackage{graphicx}
\usepackage{pgf,tikz}
\usepackage{mathrsfs}
\usetikzlibrary{arrows}
\usepackage{xspace}
\usepackage[all]{xy}
\usepackage{bm}
\usepackage{scalerel,stackengine}
\stackMath
\newcommand\reallywidehat[1]{%
\savestack{\tmpbox}{\stretchto{%
  \scaleto{%
    \scalerel*[\widthof{\ensuremath{#1}}]{\kern-.6pt\bigwedge\kern-.6pt}%
    {\rule[-\textheight/2]{1ex}{\textheight}}
  }{\textheight}%
}{0.5ex}}%
\stackon[1pt]{#1}{\tmpbox}%
}

\newcommand*{\mydprime}{^{\prime\prime}\mkern-1.2mu}

\newcommand{\A}{\mathbf A}
\newcommand{\B}{\mathbf B}
\newcommand{\C}{\mathbf C}

\newcommand{\CC}{\mathbf{C}}
\newcommand{\NN}{\mathbf{N}}

\newcommand{\ZZ}{\mathbf{Z}}
\newcommand{\RR}{\mathbf{R}}

\newcommand{\QQ}{\mathbf{Q}}
\newcommand{\KK}{\mathbf{K}}

\newcommand{\sU}{\mathscr{U}}

\newcommand{\sinf}{\mathcal{S}_\infty}

\newcommand{\se}{\subseteq}

\newcommand{\action}{\curvearrowright}

\newcommand{\da}[1]{{#1\!\downarrow}}

\DeclareMathOperator{\supp}{Supp}

\DeclareMathOperator{\Aut}{Aut}
\DeclareMathOperator{\Int}{Int}
\DeclareMathOperator{\Homeo}{Homeo}

\DeclareMathOperator{\Ends}{Ends}

\DeclareMathOperator{\Sym}{Sym}

\DeclareMathOperator{\Br}{Br}

\DeclareMathOperator{\Fix}{Fix}
\DeclareMathOperator{\Stab}{Stab}

\DeclareMathOperator{\CLO}{CCLO}
\DeclareMathOperator{\LO}{LO}

\theoremstyle{plain}
\newtheorem{thm}{Theorem}[section]
\newtheorem*{thm*}{Theorem}
\newtheorem{lem}[thm]{Lemma}
\newtheorem{prop}[thm]{Proposition}
\newtheorem{cor}[thm]{Corollary}
\theoremstyle{definition}
\newtheorem*{defn*}{Definition}
\newtheorem{defn}[thm]{Definition}
\newtheorem{example}[thm]{Example}
\newtheorem*{example*}{Example}
\newtheorem{rem}[thm]{Remark}

\newtheorem*{rem*}{Remark}

\begin{document}
\title{Topological properties of Wa\.zewski dendrite groups}
\author[B. Duchesne]{Bruno Duchesne}
\address{Institut Élie Cartan, UMR 7502, Université de Lorraine et CNRS, Nancy, France.}
\date{March 2019}
\begin{abstract}
Homeomorphism groups of generalized Wa\.zewski dendrites act on the infinite countable set of branch points of the dendrite and thus have a nice Polish topology. In this paper, we study them in the light of this Polish topology. The group of the universal Wa\.zewski dendrite $D_\infty$ is more characteristic than the others because it is the unique one with a dense conjugacy class.

For this group $G_\infty$,  we explore and prove some of its topological properties like the existence of a comeager conjugacy class, the Steinhaus property, automatic continuity, the small index subgroup property and characterization of the topology. 

Moreover, we describe the universal minimal flow of $G_\infty$ and of point-stabilizers. This allows us to prove that point-stabilizers in $G_\infty$ are amenable and to give a simple and completely explicit description of the universal Furstenberg boundary of $G_\infty$.
\end{abstract}
\thanks{The author is partially supported  by projects ANR-14-CE25-0004 GAMME and ANR-16-CE40-0022-01 AGIRA}
\thanks{This paper greatly benefited  from the participation to the conference \textit{Geometry and Structure of Polish Groups} held in Casa Mathemática Oaxaca, Mexico in June 2017. The author thanks the organizers for the invitation and people he had pleasure to speak with and who help to improve this paper: E. Glasner, A. Kwiatkowska, F. Le Maître, J. Melleray, L. Nguyen Van Thé, C. Rosendal,  T. Tsankov and P. Wesolek.}
\thanks{This paper would not have existed without previous works with N. Monod. It is pleasure to thank him for insightful discussions.}
\maketitle
\section{Introduction}
A \emph{dendrite} is a continuum (i.e. a connected metrizable compact space) that is locally connected and such that any two points are connected by a unique arc (see \cite{Nadler} for background on continua and dendrites). The group of a dendrite is merely its homeomorphism group. Dendrites are tame topological spaces that appear in various domains as Berkovich projective line or Julia sets for examples. Groups acting by homeomorphisms on dendrites share some properties  with groups acting by isometries on $\RR$-trees (see e.g. \cite{DM_dendrites}) but some dendrite group properties are very far from properties of groups acting by isometries on $\RR$-trees, for example some have the fixed-point property for isometric actions on Hilbert spaces (the so-called property (FH)).

In \cite{DM_dendritesII}, some structural properties of dendrite groups were studied and it was observed that two natural topologies on dendrite groups actually coincide. If $X$ is a dendrite without free arc (i.e. any arc contains a branch point) then the uniform convergence on $X$ and the pointwise convergence on the set of branch points $\Br(X)$ yield the same topology on $\Homeo(X)$. Since  $\Br(X)$ is countable, this yields a topological embedding  $$\Homeo(X)\to\sinf$$ where $\sinf$ is the group of all permutations of the integers with its Polish topology, which is given by the pointwise convergence. The image of this embedding being closed, this means that $\Homeo(X)$ is a \emph{non-archimedean} Polish group and it becomes natural to discover which topological properties this group enjoys. For a nice survey on topological and dynamical properties of non-archimedean groups, we refer to  \cite{MR3469133}.

For a non-empty subset $S\subset\overline{\NN}_{\geq3}=\{3,4,5,\dots,\infty\}$, the \emph{generalized  Wa\.zewski dendrite} $D_S$ is the unique (up to homeomorphism) dendrite such that any branch point of $D_S$ has order in $S$ and for all $n\in S$, the set of points of order $n$ is arcwise-dense (i.e. meets the interior of any non-trivial arc). We denote $G_S=\Homeo(D_S)$ and if $S=\{n\}$, we simply denote $D_n$ and $G_n$ for the dendrite and its group. These dendrites $D_S$ are very homogeneous, for example, the closure of any connected open subset of $D_S$ is homeomorphic to $D_S$ itself \cite[Lemma 2.14]{DM_dendritesII}.\\

\subsection{Generic elements} The aim of this paper is to study some topological properties of the Polish group $G_S$ (endowed the non-archimedean topology described above).  Let us first start with a proposition that separates dramatically $D_{\infty}$ from the others Wa\.zewski dendrites.
 
 \begin{prop} The Polish group  $G_S$ has a dense conjugacy class if and only if $S=\{\infty\}$.
\end{prop}

So, this shows that $G_\infty$ is remarkable among groups of Wa\.zewski dendrites and the remaining of the paper is essentially devoted to $G_\infty$. 

An element in a Polish group is \emph{generic} if its conjugacy class is comeager, that is, contains a countable intersection of dense open subsets. The Polish group  $G_\infty$ has generic elements. This property is sometimes called the \emph{Rokhlin property} \cite{MR2353899}.

\begin{thm}\label{ccc}The Polish group $G_\infty$ has a comeager conjugacy class.
\end{thm}

Our proof of this theorem relies on Fraïssé theory and $G_\infty$ appears as the automorphism group of some Fraïssé structure. In Section \ref{fraisse}, we detail this Fraïssé structure and some of the properties needed to prove Theorem~\ref{ccc} relying on results in \cite{MR1162490,Kechris-Rosendal}.

\subsection{Automatic continuity}A Polish group has the \emph{automatic continuity} property if any abstract group homomorphism to any separable topological group is actually continuous. This property is quite common for large Polish groups and we refer to   \cite{rosendal2009automatic} for a survey. Automatic continuity is a consequence of the following property.

\begin{defn}\label{Stein}A topological group $G$ has the \emph{Steinhaus property} if there is $k\in\NN$ such that for any symmetric subset $W\subset G$ such that there is $(g_n)_{n\in\NN}$ with
$$\bigcup_{n\in\NN}g_nW=G$$
then $W^k$ is a neighborhood of the identity.
\end{defn}

\begin{thm}\label{144-Steinhaus}The Polish group $G_\infty$ has the Steinhaus property.
\end{thm}

In particular, we obtain the following corollary (see \cite[Proposition 2]{Rosendal-Solecki}).
\begin{cor}\label{autcont}The Polish group $G_\infty$ has the automatic continuity property.
\end{cor}

It is well known that automatic continuity implies the uniqueness of the Polish group topology. So we have another proof of a particular case of a result due to Kallman. Actually, the uniqueness of the Polish group topology for $G_S$ (with any $S\subseteq\overline{\NN}_{\geq3}$) is  a direct application of \cite[Theorem 1.1]{MR831205}. So, we can speak about the Polish topology on $G_S$ without any ambiguity.
 
 \begin{rem} The proofs use intrinsically that $\Aut(\QQ,<)$ is Steinhaus. Moreover, $G_\infty$ has the Bergman property (any isometric action on a metric space has bounded orbits) but contrarily to $\Aut(\QQ,<)$ \cite[Corollary 7]{Rosendal-Solecki}, $G_\infty$ is far to have the fixed point property for (non-necessarily continuous) actions on compact metrizable spaces. For example, the action of $G_\infty$ on the dendrite $D_\infty$ is minimal \cite{DM_dendritesII}.
 \end{rem}
 
 Corollary \ref{autcont} means that the Polish topology on $G_\infty$ is maximal among separable group topologies on this space. The following theorem goes in the other direction and shows that the Polish topology on $G_S$ is a least element among Hausdorff group topologies on $G_S$ for any $S$. See Section \ref{umin} for details.

 \begin{thm}\label{tumin}For any $S\subset\overline{\NN}_{\geq3}$, the Polish group $G_S$ is universally minimal.
 \end{thm}
Combining Corollary \ref{autcont} and Theorem \ref{tumin}, we obtain the following characterization of the Polish topology on $G_\infty$.
 \begin{cor}There is a unique separable Hausdorff group topology on $G_\infty$.
 \end{cor}
 If a Polish group has \emph{ample generics} then it has the Steinhaus property. So, exhibiting ample generics is a common way to prove the  Steinhaus property (see \cite[\S1.6]{Kechris-Rosendal}). In our situation, this is not possible.
 
 \begin{prop}\label{nag} The Polish group $G_\infty$ does not have ample generics.\end{prop}
 Actually, the diagonal action of $G_\infty$ on $G_\infty\times G_\infty$ by conjugation does not have any comeager orbit. Our proof relies on the same result for $\Aut(\QQ,<)$ due to Hodkinson (see \cite{MR2354899}). 
 
 \subsection{Small index property} A Polish group has the \emph{small index property} if any subgroup of small index, i.e. of index less than  $2^{\aleph_0}$, is open. 
 
 \begin{thm}\label{sis}The Polish group $G_\infty$ has the small index property.
 \end{thm}
By definition of the topology of pointwise convergence on branch points, a basis of neighborhoods of the identity is given by pointwise stabilizers of finitely many branch points. The number of  branch points being countable, these subgroups have countable index. So, Theorem \ref{sis} shows that subgroups of small index contains the pointwise stabilizer of some finite set of branch points and thus have countable index. 

Let us point out that the property that subgroups of countable index are open is equivalent to the fact that any homomorphism to $\mathcal{S}_\infty$ is continuous. This last property is a particular case of the automatic continuity property. 

This theorem enlightens the idea that the Polish topology on $G_\infty$ is indeed an algebraic datum: it can be recovered by subgroups of small index.

\subsection{Universal minimal flows}The group $G_\infty$ is the automorphism group of a countable structure, the set of branch points with the betweenness relation, and it is also a group of dynamical origin since it comes with its action on the compact space $D_\infty$. So, it is natural to try to understand possible $G_\infty$-\emph{flows}, that are continuous actions of $G_\infty$ on compact spaces.
\begin{rem}The group of homeomorphisms of a  metrizable compact space (endowed with the topology of uniform convergence) is separable. So, Theorem \ref{autcont} implies that any action of $G_\infty$ on a  metrizable compact space by homeomorphisms is actually continuous.\end{rem}
Let $G$ be a topological group. A $G$-flow is \emph{minimal} if every orbit is dense. It is  a remarkable fact there is a minimal $G$-flow which has the following universal property: Any other minimal $G$-flow is a continuous equivariant image of this largest flow. It is called the \emph{universal minimal flow} of $G$ (see for example \cite{GlasnerLNM} for details).

Usually this universal $G$-flow is very large and not explicit at all. For $G_\infty$, we identify this universal minimal flow with a subset of the compact space of  linear orders on the set of branch points. This subset consists of linear orders that are  \textit{converging and convex}. They reflect the \textit{dendritic} nature of $D_\infty$. We refer to Section~\ref{umf} for definitions.
\begin{thm} The universal minimal flow of $G_\infty$ is the set of convex converging linear orders on the set of branch points of $D_\infty$.
\end{thm}
During this work, a more general result was proved in  \cite{kwiatkowska2018universal} but the description is a bit different. Our point is to show that the stabilizer of some generic converging convex linear order is actually \emph{extremely amenable}, i.e. has the fixed point property on compacta. Following \cite{MR2140630}, this is equivalent to the Ramsey property for the underlying structure. For this Ramsey property, we rely on \cite{MR3436366}.

We also obtain a description of the universal minimal flow of end stabilizers and this knowledge allows us to obtain amenability results in  a non usual way. Let us observe that for a locally finite tree, the amenability of stabilizers of vertices or end points is easy but for dendrites it is not clear if stabilizers of points are amenable in general.

\begin{thm}\label{ufb} For any point $x$ in $D_\infty$, the stabilizer of $x$ in $G_\infty$ is an amenable topological group.\end{thm}

Conversely, it is known that an amenable group acting continuously on a dendrite stabilizes a subset with at most two points \cite{ShiYe2016}.

This amenability result allows to identify the universal Furstenberg boundary of $G_\infty$, that is the universal strongly proximal minimal $G_\infty$-flow. Let $\xi$ be some end point of $D_\infty$ and $G_\xi$ its stabilizer. Let us denote by $\widehat{G_\infty/G_\xi}$ the completion of $G_\infty/G_\xi$ for the uniform structure coming from the right uniform structure on $G_\infty$. We obtain a first description of the universal Furstenberg boundary.

\begin{thm}\label{ufba} The universal Furstenberg boundary of $G_\infty$ is $\widehat{G_\infty/G_\xi}$.
\end{thm}

Even if the set of end points is a dense $G_\delta$-orbit in $D_\infty$, the natural map $\widehat{G_\infty/G_\xi}\to D_\infty$ is not a homeomorphism (Proposition~\ref{nothomeo}) and thus $D_\infty$ is a Furstenberg boundary of $G_\infty$ but not the universal one. This result should be compared to the fact that the universal Furstenberg boundary of $\Homeo(\mathbf{S}^1)$ is $\mathbf{S}^1$ itself.\\

At the end of this paper, we give another description of this universal Furstenberg boundary. It appears as a closed subset of a natural countable product of totally disconnected compact spaces. The description is simple and shows that it is a countable collection of $G_\infty$-orbits. This universal Furstenberg boundary is the space $K$ that appears in  Subsection~\ref{isomorphic}.

\setcounter{tocdepth}{1}
\tableofcontents

\section{Wa\.zewski dendrites and their homeomorphism groups}\label{wazewski}
\subsection{Wa\.zewski dendrites} A \emph{dendrite} is a connected metrizable compact space that is locally connected and such that any two points $x,y$ are connected by a unique arc $[x,y]$. Simple examples are given by compactifications of locally finite simplicial trees. Some examples are more complicated. For example, the Julia set of the polynomial map $z\mapsto z^2+i$ of the complex line $\CC$ is a dendrite (See Figure \ref{Julia}). A \emph{subdendrite} is a closed and connected subset $S$ of a dendrite $D$. It is a dendrite on its own and there is a retraction $\pi_S\colon D\to S$ such that for any $x\in D$ and $y\in S$, $\pi_S(x)\in[x,y]$. This retraction is also called the \emph{first-point map} to $S$.

\begin{figure}[b]
\includegraphics[width=.8\textwidth]{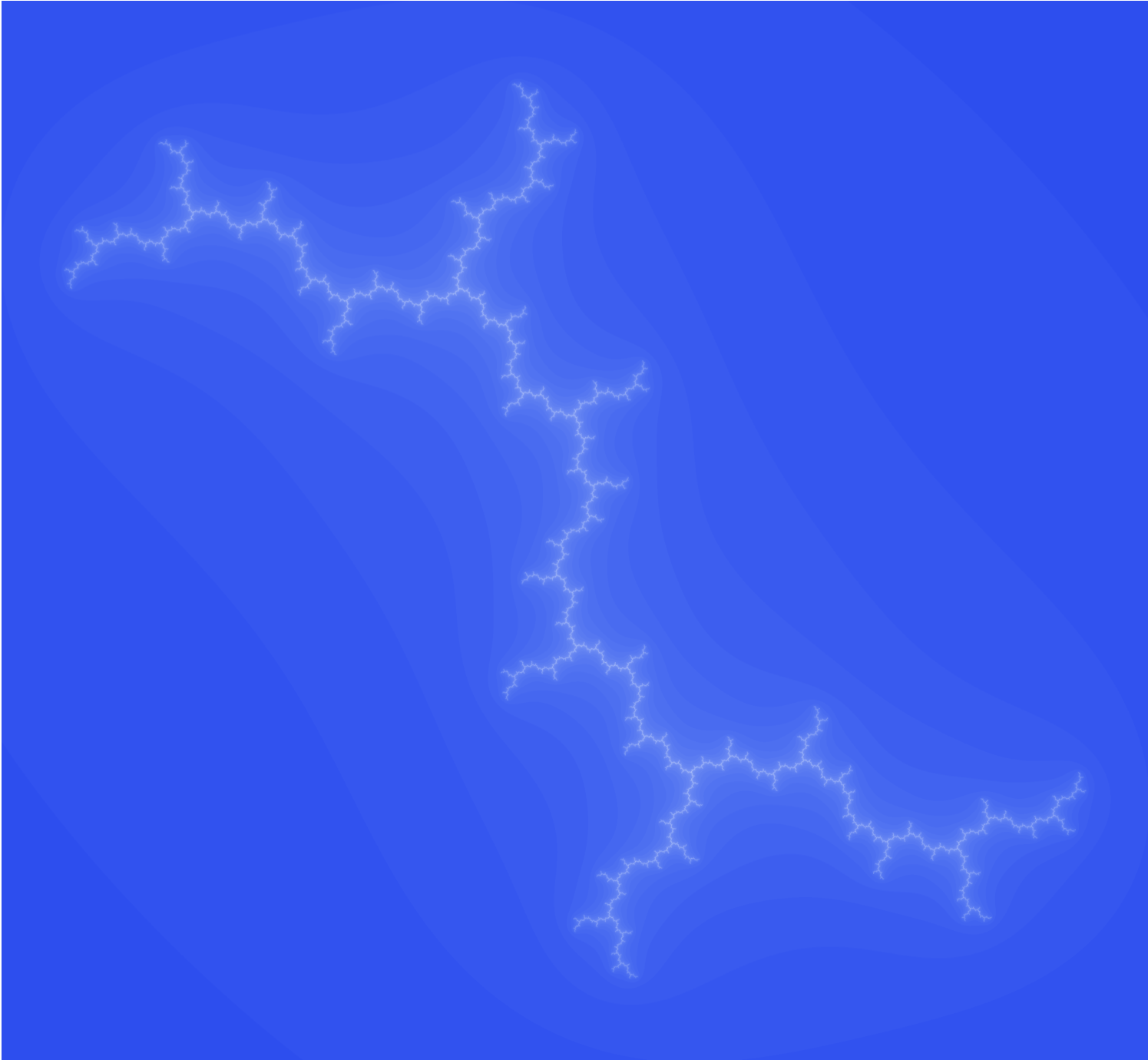}
\caption{The Julia set of $z\mapsto z^2+i$. Picture realized with \textit{Mathematica}.}\label{Julia}
\end{figure}
 In a dendrite $X$, there are three types of points $x\in X$, according to the cardinal of $\mathcal{C}_x$, the space of connected components of $X\setminus\{x\}$. This number is at most countable and is called the \emph{order} of $x$.

\begin{itemize}
\item If the complement $X\setminus\{x\}$ remains connected, $x$ belongs to the set $\Ends(X)$ of \emph{end points}.
\item If $x$ separates $X$ into two components, it is a \emph{regular point}.
\item Otherwise, it is at least 3 and $x$ belongs to the set $\Br(X)$ of \emph{branch points}.
\end{itemize}

Let $X$ be a dendrite and  $c\colon X^3\to X$ be the center map, that is $c(x,y,z)=[x,y]\cap[y,z]\cap[z,x]$, which is reduce to a unique point. Let us do a few observations:

\begin{itemize}
\item $c$ is symmetric,  
\item $c(x,y,z)=z$ if and only if $z\in[x,y]$,
\item if $c(x,y,z)\notin\{x,y,z\}$, $c(x,y,z)$ is a branch point 
\item and in particular, $\Br(X)$ is $c$-invariant, i.e. $c(\Br(X))=\Br(X)$. 
\end{itemize}
In Bowditch's terminology, $(X,c)$ is a median space \cite{BowditchAMS}. For two points $x\neq y$ in a dendrite $X$, we denote by $D(x,y)$ the connected component of $X\setminus\{x,y\}$ that contains $]x,y[$. Let us say that a subset $Y$ of $X$ is $c$-\emph{closed} if for any $x,y,z\in Y$, $c(x,y,z)\in Y$. For $Y\subset X$, we define the $c$-\emph{closure} of $Y$ to be $c(Y^3)$, which happens to be the smallest $c$-closed subset of $X$ containing $Y$.

\begin{lem} For any $Y\subset X$, $c(Y^3)$ is $c$-closed.
\end{lem}

\begin{proof} Let $x_1,x_2,x_3\in c(Y^3)$ and $m=c(x_1,x_2,x_3)$. If $m\in\{x_1,x_2,x_3\}$ then we are done. Otherwise, for each  $i=1,2,3$, one can find a point $y_i\in Y$ that is not in $C_{x_i}(m)$, the connected component of $X\setminus\{x_i\}$ that contains $m$. Thus $m=c(y_1,y_2,y_3)\in c(Y^3)$.\end{proof}

Let $S$ be a non-empty subset of $\overline{\NN}_{\geq3}=\{3,4,5,\dots,\infty\}$. The \emph{Wa\.zewski dendrite} $D_S$ is the unique (up to homeomorphism) dendrite such that all orders of branch points belong to $S$ and for any $n\in S$, the set of points of orders $n$ is arcwise-dense. See  \cite[\S12]{DM_dendrites} for a few historical references and a reference to the proof of this characterization. With this characterization, it easy to see that the closure of any open connected subset in $D_S$ is actually homeomorphic to $D_S$. We denote by $G_S$ the homeomorphism group of $D_S$. If $S=\{n\}$, we simply denote $D_n$ and $G_n$ for the dendrite and its group. For example,  $D_\infty$ appears to be homeomorphic to the Berkovich projective line over $\CC_p$. See \cite[Figure 1] {Hrushovski-Loeser-Poonen} for explanations and a nice picture of this dendrite.

Let us recall some properties of the groups $G_S=\Homeo(D_S)$ proved in \cite[\S 6 \& 7]{DM_dendritesII}. The first one shows how homogeneous the dendrite $D_S$ is.
To any finite subset $F$ of the dendrite $D_S$, we associate a finite vertex-labelled simplicial tree $\langle F \rangle$ as follows. The sub-dendrite $[F]$, i.e. the smallest subdendrite containing $F$,  is a finite tree in the topological sense, i.e.\ the topological realization of a finite simplicial tree. Such a simplicial tree is not unique because degree-two vertices can be added or removed without changing the topological realization. We choose for $\langle F \rangle$ to retain precisely one degree-two vertex for each element of $F$ which is a regular point of the dendrite $[F]$. Thus, $\langle F \rangle$ is a tree whose vertex set contains $F$. Finally, we label the vertices of $\langle F\rangle$ by assigning to each vertex its order in $D_S$.

\begin{prop}[{\cite[Proposition 6.1]{DM_dendritesII}}]\label{prop:extension}
Given two finite subsets $F, F'\se D_S$, any isomorphism of labelled graphs $\langle F \rangle \to \langle F' \rangle$ can be extended to a homeomorphism of $D_S$.
\end{prop}

This simple proposition have strong corollaries for $G_S$. For example, $G_S$ acts 2-transitively on the set of points of a given order. Moreover, $G_S$ is a simple group and  if $S$ is finite then the action of $G_S$ on the set of branch points (which is countable) is oligomorphic \cite[Corollary 6.7]{DM_dendritesII}. For an introduction to oligomorphic groups, we refer to \cite{MR2581750}.

The structure of the group $G_S$ completely determines the dendrite $D_S$: if $G_S$ and $G_{S’}$ are isomorphic then $S=S’$ and an automorphism of $G_S$ is always given by a conjugation \cite[Corollary 7.4 \& 7.5]{DM_dendritesII}.

Since the action on the set of branch points $\Br(D_S)$ completely determines the action, $G_S$ embeds as a closed subgroup of $\mathcal{S}_\infty$. With this topology, if $S$ is finite then the Polish group $G_S$ has the strong Kazhdan property (T) \cite[Corollary 6.9]{DM_dendritesII}. Without the assumption of finiteness of $S$, the discrete group $G_S$ has Property (OB) (every action  by isometries on a metric space has bounded orbits) \cite[Corollary 6.12]{DM_dendritesII}. \\

We will need to construct global homeomorphisms from  patches of partial homeomorphisms. This is possible thanks to the following lemma \cite[Lemma 2.9]{DM_dendritesII}.
\begin{lem}[Patchwork lemma]\label{lem:patchwork}
Let $\sU$ be a family of disjoint open connected subsets of a dendrite $X$ and let $(f_U)_{U\in \sU}$ be a family of homeomorphisms $f_U \in\Homeo(U)$ for $U\in\sU$. Suppose that each $f_U$ can be extended continuously to the closure $\overline U$ by the identity on the boundary $\overline U\setminus U$.

Then the map $f\colon X\to X$ given by $f_U$ on each $U\in \sU$ and the identity everywhere else is a homeomorphism.
\end{lem}

\subsection{Dynamics of individual elements}

Let $g$ be a homeomorphism of a dendrite $X$. An arc $[x,y]\subset X$ is \emph{austro-boreal} for $g$ if $x\neq y$ are fixed and there is no fixed-point in $]x,y[$. Observe that in this case, the restriction of the action of $g$ on $]x,y[$ is conjugated to an action by a non-trivial translation on the real line.

For a non-trivial arc $[x,y]$, we denote by $D(x,y)$ the connected component of $X\setminus\{x,y\}$ that contains $]x,y[$. We denote by $D(g)$ the union of all $D(I)$ where $I$ is an austro-boreal arc for $g$ and by $K(g)$ its complement in $X$.

The following proposition \cite[Proposition 10.6]{DM_dendrites} describes the dynamics of a homeomorphism of a dendrite.

\begin{prop}\label{prop:indiv}
The decomposition $X=D(g) \sqcup K(g)$ has the following properties.
\begin{enumerate}[label=(\roman*)]
\item $D(g)$ is a (possibly empty) open $g$-invariant set on which $g$ acts properly discontinuously. In particular, $K(g)$ is a non-empty compact $g$-invariant set.\label{pt:tectonic:dec}
\item $K(g)$ is a disjoint union of subdendrites of $X$. Moreover, $g$ preserves each such subdendrite and has a connected fixed-point set in each.\label{pt:tectonic:K}
\end{enumerate}
\end{prop}

The subdendrites that appear in (ii) in the above proposition are actually the connected component of $K(g)$. Let us precise the action of $g$ on these connected components.
 
\begin{lem}\label{rem:kg}If $C$ is a connected component of $K(g)$,  then $g$ permutes the connected components of $C\setminus\Fix(g)$ where $\Fix(g)$ is the set of fixed points of $g$. Moreover none of these connected components of $C\setminus\Fix(g)$ is invariant.
\end{lem}

\begin{proof} The connected component $C$ is $g$-invariant by Proposition~\ref{prop:indiv} and contains at least one fixed point by the fixed point property for dendrites (See for example \cite[Lemma 2.5]{DM_dendrites}). So $g$ permutes the connected components of $C\setminus\Fix(g)$. Let $C'$ be a connected component of 
$C\setminus\Fix(g)$. Let $x\in C'$ and $y\in\overline{C’}$ be distinct points. There is a sequence $(x_n)$ converging to $y$ such that $x_n\in C'$ for all $n\in\NN$. Since $C'$ is connected $[x,x_n]\subset  C'$ for all $n\in\NN$. For all $z\in ]x,y[$ and $n$ large enough, $z\in [x,x_n]$. Thus $[x,y[\subset C'$. For $y' \in\overline{C'}\setminus C'$ such that $y’\neq y$, the point $c(x,y,y')\in C'$ separates $y$ and $y'$. Since $\Fix(g)\cap C$ is connected then $\overline{C'}$ contains exactly one fixed point. This fixed point is an end point of the subdendrite $\overline{C'}$ because $C’$ is connected. Assume $C'$ is $g$-invariant, then by \cite[Lemma 4.8]{DM_dendritesII}, $g$ has two fixed points  in $\overline{C'}$ and we have a contradiction. \end{proof}

Let $X$ be a dendrite and $g\in\Homeo(X)$. It will be useful for us to decide where a point $x\in X$ lies  in the decomposition $X=K(g)\sqcup D(g)$ from Proposition \ref{prop:indiv}, using only finitely many points in the $g$-orbit of $x$.

\begin{lem}\label{lem:c}Let $g\in\Homeo(X)$ and $x\in X$. Then,  
\begin{itemize}
\item $x$ is in the interior of some austro-boreal arc if and only if $g(x)\in]x,g^2(x)[$,
\item $x\in D(g)$ if and only if $c(g(x),g^2(x),g^3(x))\in]c(x,g(x),g^2(x)),c(g^2(x),g^3(x),g^4(x)[$,
\item $x\in K(g)$ if and only if $[x,g(x)]\cap\Fix(g)\neq\emptyset$. 
\end{itemize}
\end{lem}

\begin{proof} Let us start with elements in the interior of some austro-boreal arc. Let $[y,z]$ be some austro-boreal arc for $g$. If $x\in]y,z[$ then $g(x)\in]x,g^2(x)[$ because the action of $g$ on $]y,z[$ is conjugated to an action on the real line by translation. Conversely, if $g(x)\in]x,g^2(x)[$ then for any $n\in\ZZ$, $g^n(x)\in]g^{n-1}x,g^{n+1}(x)[$ and thus $\overline{\cup_{n\in\ZZ}[g^n(x),g^{n+1}(x)]}$ is an austro-boreal arc.

If $x\in D(y,z)$ for some austro-boreal arc $[y,z]$ then let $p$ be the image of $x$ under the first-point to $[y,z]$. The point $p$ lies in $]y,z[$ then $gp=c(x,g(x),g^2(x))\in ]y,z[$ and thus $gc(x,g(x),g^2(x))\in[c(x,g(x),g^2(x)),g^2c(x,g(x),g^2(x))]$ that is $c(g(x),g^2(x),g^3(x))$ belongs to $]c(x,g(x),g^2(x)),c(g^2(x),g^3(x),g^4(x)[$. Conversely, if $$c(g(x),g^2(x),g^3(x))\in]c(x,g(x),g^2(x)),c(g^2(x),g^3(x),g^4(x)[,$$ then by the first part this means that $c(g(x),g^2(x),g^3(x))$ lies in some austro-boreal arc. Let $y,z$ be the ends of this arc. Then, by construction, $x\in D(y,z)$.

Let $x\in K(g)$, let $C$ be its connected component in $K(g)$ and let $C_0$ be its connected component in $C\setminus\Fix(g)$. By Lemma~\ref{rem:kg} $gC_0\cap C_0=\emptyset$ and thus there is a fixed point $p$ in $[x,g(x)]$. Conversely, if there is a fixed point $p$ in $[x,g(x)]$ then $d\notin D(g)$ since the action of $g$ on $D(g)$ is properly discontinuous. 
\end{proof}

The decomposition $X=K(g)\sqcup D(g)$ is not really a group invariant for the cyclic group $\langle g\rangle$ generated by $g$. Each part is $\langle g\rangle$-invariant but the decomposition is not the same for every element of $\langle g\rangle$. Let us illustrate this phenomenon.

\begin{example} Let us fix $\xi_\pm$ two end points of the Wa\.zewski dendrite $D_3$. Let $x$ be some regular point of $D_3$ and let $C_1,C_2$ be the two connected components of $D_3\setminus\{x\}$. Let $\varphi_i$ be an homeomorphism from $D_3\setminus\{\xi_+\}$ to $C_i$. Let $\gamma$ be some homeomorphism of $D_3$ such that $[\xi_-,\xi_+]$ is austro-boreal for $\gamma$. We define an homeomorphism $g$ of $D_3$ fixing $x$ and such that $g|_{C_1}=\varphi_2\circ\varphi_1^{-1}$ and $g|_{C_2}=\varphi_1\circ\gamma\circ\varphi_2^{-1}$. The map $g$ is well-defined thanks to Lemma~\ref{lem:patchwork}. By construction, we have $K(g)=D_3$ and $D(g)=\emptyset\ $ but $K(g^2)=\{x\}$ and $D(g^2)=C_1\cup C_2$. \end{example}

Nonetheless, we have the following inclusions.

\begin{lem}\label{KD}Let $X$ be a dendrite and $g\in\Homeo(X)$. For any $n\in\NN$, $D(g)\subset D(g^n)$ and $K(g^n)\subset K(g)$.
\end{lem}

\begin{proof} It suffices to prove the first inclusion and pass to the complement to get the other one. If an arc is austro-boreal for $g$, it is austro-boreal for any of its non-trivial power and thus $D(g)\subset D(g^n)$. 
\end{proof}

 \section{Fraïssé theory and generic elements}\label{fraisse} We use the notations of \cite{Kechris-Rosendal} and denote by $\KK$ the Fraïssé structure associated to the action of $G_S$ on the countable set $\Br(D_S)$ (see \cite[\S 1.2]{Kechris-Rosendal}) and by $\mathcal{K}$ the Fraïssé class of finite substructures of $\KK$. In particular, $\KK$ is the Fraïssé limit of $\mathcal{K}$ and $\Aut(\KK)=G_S$. Let us briefly explain what is this structure. The structure $\KK$ is the set $\Br(D_S)$ with the all relations $R_{i,n}\subset \Br(D_S)^n$ given by orbits of the diagonal actions of $G_S$ on $\Br(D_S)^n$. 
 
 Let us briefly recall what it means to be a Fraïssé class. The class $\mathcal{K}$ is a countable class of finite structures over some fixed countable signature that enjoys the following properties:
 \begin{enumerate}
 \item\textbf{Hereditary property.} For any $\mathbf{B}\in\mathcal{K}$ and $\mathbf{A}\leq\mathbf{B}$ (i.e. $\mathbf{A}$ can be embedded in $\mathbf{B}$), $\mathbf{A}\in\mathcal{K}$.
 \item\textbf{Joint embedding property.} For any $\mathbf{A},\mathbf{B}\in\mathcal{K}$, there is $\mathbf{C}$ such that $\mathbf{A},\mathbf{B}\leq\mathbf{C}$.
 \item\textbf{Amalgamation property. }For $\mathbf{A},\mathbf{B},\mathbf{C}\in\mathcal{K}$, if $f\colon\mathbf{A}\to\mathbf{B}$ and $g\colon\mathbf{A}\to\mathbf{C}$ are embeddings then there is $\mathbf{D}\in\mathcal{K}$ and embeddings $r\colon\mathbf{B}\to\mathbf{D},\ s\colon\mathbf{C}\to\mathbf{D}$ such that $r\circ f=s\circ g$.
 \end{enumerate}
 The structure $\mathbf{K}$, the Fraïssé limit of $\mathcal{K}$, is a countable structure over the same signature such that any finite substructure belongs to $\mathcal{K}$ and $\mathbf{K}$ is ultra-homogeneous: any partial isomorphism between finite substructures extends to a global isomorphism.

 \begin{rem}\label{structure} The structure $\KK$ is, a priori, given by $\Br(D_S)$ (as underlying set) and infinitely many relations corresponding to orbits in $\Br(D_S)^n$ for $n\in\NN$. But actually, $G_S$ can be realized as the automorphism group of a structure given by $\Br(D_S)$ and a unique relation: the betweenness relation $B$ where $B(z;x,y)\iff z\in[x,y]$. A bijection of $\Br(D_S)$ that preserves the betweenness relation is actually given by a homeomorphism of $D_S$ \cite[Proposition 2.4]{DMW}.  
 
A betweenness relation $B$ is of \emph{positive type} if for any $x,y,z$ there is $w$ such that $B(w;x,y)\wedge B(w;y,z)\wedge B(w;z,x)$. Moreover, a finite set with a betweenness relation with positive type (as it is the case for finite subsets of $\Br(D_\infty)$ closed under the center map) has a tree structure \cite[Lemma 29.1]{MR1388893} and thus embeds in the set of branch points of the universal dendrite $D_\infty$. We refer to \cite{MR1388893} for details about betweenness relations and being of positive type.  By Proposition \ref{prop:extension}, any isomorphism between two finite subtree of $\Br(D_\infty)$ can be realized as the restriction of some element of $G_\infty$. This way $\Br(D_\infty)$ with the betweenness relation is the Fraïssé limit of the class of finite betweenness structures with positive type. 
 
Let us observe that the center map can defined only in terms of the betweenness relation. Actually, $c(x,y,z)=w$ is equivalent to $B(w;x,y)\wedge B(w;y,z)\wedge B(w;z,x)$.\end{rem}
 
 We also denote by $\mathcal{K}_p$ the class of systems $\mathcal{S}=\langle \mathbf{A},\varphi\colon\mathbf{B}\to\mathbf{C}\rangle$ where $\mathbf{B},\mathbf{C}\subseteq \mathbf{A}$ are finite substructures of $\KK$ and $\varphi$ is an isomorphism between these substructures. Let $\mathcal{S}=\langle \mathbf{A},\varphi\colon\mathbf{B}\to\mathbf{C}\rangle$ and $\mathcal{T}=\langle \mathbf{D},\psi\colon\mathbf{E}\to\mathbf{F}\rangle$ be two systems of $\mathcal{K}_p$. An \emph{embedding} of $\mathcal{S}$ into $\mathcal{T}$ is an embedding of structures $f\colon \mathbf{A}\to\mathbf{D}$ that induces an embedding of $\mathbf{B}$ in $\mathbf{E}$, an embedding of $\mathbf{C}$ in $\mathbf{F}$ and such that $f\circ\varphi\subseteq \psi\circ f$. In that case, we also say that $\mathcal{T}$ is an \emph{extension} of $\mathcal{S}$. This notion of embeddings allows us to speak about the joint embedding property (JEP) or the amalgamation property (AP) for $\mathcal{K}_p$. A subclass $\mathcal{L}$ of $\mathcal{K}_p$ is \emph{cofinal} if for any system $\mathcal{S}\in\mathcal{K}_p$, there is $\mathcal{T}\in\mathcal{L}$ and an embedding of $\mathcal{S}$ into $\mathcal{T}$.
 For a system $\mathcal{S}=\langle \mathbf{A},\varphi\colon\mathbf{B}\to\mathbf{C}\rangle$ and $g\in G_S$, we say that $g$ \emph{induces} $\varphi$ if there is $A\subset \Br(D_S)$ and an isomorphism $f\colon \A\to A$ such that $\varphi=f^{-1}gf$. In this case, by an abuse of notation, we consider $\A$ as a subset of $\Br(D_S)$ and forgot about $f$.
 


\subsection{Existence of a dense conjugacy class}
The following proposition shows that $G_\infty$ is remarkable amongst all the Wa\.zewski groups. 
\begin{prop}\label{prop:jep} The Polish group $G_S$ has a dense conjugacy class if and only if $S=\{\infty\}$.
\end{prop}

\begin{proof} Thanks to \cite[Theorem 1.1]{Kechris-Rosendal}, it suffices to show that $\mathcal{K}_p$ satisfies (or not) the joint embedding property. 

Assume that $S$ contains $n\neq\infty$. Choose a point $x\in D_S$ of order $n$ and $x_1,\dots,x_n$ in distinct connected components of $D_S\setminus\{x\}$ such that there exists $g\in G_S$ with $gx_i=x_{i+1}$ ($i\in\ZZ/n\ZZ$). We set $\mathbf{A}=\mathbf{B}=\mathbf{C}=\{x,x_1,\dots,x_n\}$, $\varphi$ to be the restriction of $g$ on $\mathbf{B}$ and $\mathcal{S}=\langle \mathbf{A},\varphi\colon\mathbf{B}\to\mathbf{C}\rangle$. Let $\mathcal{T}=\langle \mathbf{D},\psi\colon\mathbf{E}\to\mathbf{F}\rangle$ where $\mathbf{D}=\mathbf{E}=\mathbf{F}$ are two points and $\psi$ is the identity. Now $\mathcal{S}$ and $\mathcal{T}$ do not have a joint  embedding because any extension of $\varphi$ in $\Homeo(D_S)$ has a unique fixed point, namely $x$.

Now, assume $S=\{\infty\}$. We claim that a partial isomorphism between finite substructures can be extended to a homeomorphism of $D_\infty$ fixing a branch point and fixing pointwise a connected component of the complement of this fixed point. 

Assume the claim holds true. Let $\mathcal{S}=\langle \mathbf{A},\varphi\colon\mathbf{B}\to\mathbf{C}\rangle$ and $\mathcal{T}=\langle \mathbf{D},\psi\colon\mathbf{E}\to\mathbf{F}\rangle$ be elements of $\mathcal{K}_p$. Thanks to the claim, we assume that $\varphi$ is induced by $f\in G_\infty$ that fixes a point $x\in\Br(D_\infty)$ and $\psi$ is induced by $g\in G_\infty$ that fixes a point $y\in\Br(D_\infty)$. Moreover, thanks to the disjunction lemma \cite[Lemma 4.3]{DM_dendrites}, we may assume that $y$ (resp. $x$) is in a component of $\mathcal{C}_x$ (resp. $\mathcal{C}_y$) pointwise fixed by $f$ (resp. $g$). Now define $h$ to be the identity on $D(x,y)$, acts like $f$ on the support of $f$ and like $g$   on the support of $g$. This $h$ yields a joint extension of $\varphi$ and $\psi$.

Let us prove the claim. Any $g\in G_\infty$ has a fixed point $x\in D_\infty$ and thus permutes the components of $D_\infty\setminus\{x\}$. These components are all homeomorphic to $D_\infty\setminus\{\xi\}$ where $\xi$ is some end point. If $x$ is a branch point, we may glue a new copy $D$ of $D_\infty$ by identifying some end point  in $D$ with  $x$. If $x$ is not a branch point, we glue countably many copies of $D_\infty$. The new dendrite is $D_\infty$ once again and $x$ is a branch point. One obtain a new dendrite homeomorphic to $D_\infty$ and one can extend $g$ by the identity on the new copies of $D_\infty$.\end{proof}

\subsection{Existence of a comeager conjugacy class} Let us recall that for a dendrite $X$ and two points $x,y\in X$, we denote by $D(x,y)$ the connected component of $X\setminus\{x,y\}$ that contains $]x,y[$. 

For the remaining of this section, we consider only the dendrite $D_\infty$ and its associated Fraïssé class $\mathcal{K}$. In \cite{MR1162490}, Truss introduced a general way to prove existence of generic elements in automorphism groups of countable structures. To prove this existence, it suffices to show that $\mathcal{K}_p$ has the joint embedding property (JEP) and the amalgamation property (AP) defined above. Actually a cofinal version of (AP) is sufficient. In \cite{Kechris-Rosendal}, a weaker condition, the weak amalgamation property (WAP) has been shown to be the necessary and sufficient amalgamation condition. 

\begin{rem}\label{notAP} The class $\mathcal{K}_p$ does not have (AP). Let us consider the simple example $\mathcal{S}=\langle \mathbf{A},\varphi\colon\mathbf{B}\to\mathbf{C}\rangle$ where $x,y$ are two  distinct points of $\Br(D_\infty)$, $B=\{x\}$, $C=\{y\}$, $A=\{x,y\}$ and $\varphi(x)=y$. Actually, $\varphi$ can be realized by an automorphism $g_1$ that fixes a point $p$ in $[x,y[$ or by an element $g_2$ such that $[x,y]$ is included in some austro-boreal arc for $g_2$. If $\varphi$ is extended by $\varphi_1$ the restriction  of $
g_1$ on $\{x,p\}$ and by  $\varphi_2$ the restriction  of $g_2$ on $\{x,y\}$, it is not possible to amalgamate $\varphi_1$ and $\varphi_2$ over $\varphi$. Actually if $\psi$ is an amalgamation, it is given by an element $g\in G_\infty$ that has a fixed point in $[x,y]$ because of $\varphi_1$ and simultaneously such that $[x,y]$ is included in some austro-boreal arc for $g$ because of the first point in Lemma \ref{lem:c}. Thus we have a contradiction.\\
\end{rem}

Below, we define a subclass $\mathcal{L}$ of $\mathcal{K}_p$ for which we show cofinality and the amalgamation property. As explained in Remark~\ref{structure}, we consider the structure $\mathbf{K}$ with the betweenness relation and the associated center map $c$ and for a finite structure of positive type $\A\in\mathcal{K}$ (that is a $c$-closed subset of $\mathbf{K}$) and points $x,y,z\in\A$, we write $x\in[y,z]$ if $c(x,y,z)=x$ that is $B(x;y,z)$. For $x,y\in\A$, we define $D(x,y)=\{z\in \A, c(x,y,z)\notin\{x,y\}\}$. Let us observe that if $\A$ is embedded in $\Br(D_\infty)$ then these definitions are consistent with the ones in $D_\infty$.  For a system $\langle \mathbf{A},\varphi\colon\mathbf{B}\to\mathbf{C}\rangle\in\mathcal{K}_p$ and $x\in \B$, we write $\varphi^n(x)$ for $n\in\NN$ if $\varphi(x),\dots,\varphi^{n-1}(x)\in\B$ and define $\varphi^n(x)$ to be $\varphi(\varphi^{n-1}(x))$. In particular, when this notation is used it implies implicitly that $\varphi(x),\dots,\varphi^{n-1}(x)$ are well-defined and  belong to $\B$. If there is $n\in \NN$ such that  $\varphi^n(x)=x$ then we say that $x$ is $\varphi$-\emph{periodic}. In that case, its \emph{period} is $\inf\{n>0,\ \varphi^n(x)=x\}$. 

Let  $\mathcal{S}=\langle \mathbf{A},\varphi\colon\mathbf{B}\to\mathbf{C}\rangle\in\mathcal{K}_p$ be a system. We define a $\varphi-$\emph{orbit} to be an equivalence class under the equivalence relation on $\mathbf{B}\cup\mathbf{C}$ generated $x\sim_\varphi y\iff y=\varphi(x)$ or $x=\varphi(y)$.

\begin{defn}\label{class}We consider the subclass $\mathcal{L}\subseteq\mathcal{K}_p$ of systems  $\mathcal{S}=\langle \mathbf{A},\varphi\colon\mathbf{B}\to\mathbf{C}\rangle\in\mathcal{K}_p$ with $\A,\B$ and $\C$ of positive type  and that satisfy the following conditions. There is $\B_0\subset \B$ such that 
\begin{enumerate}
\item for any $y\in\B$, there is a $x\in\B_0$ and $p$ non-negative integer such that $y=\varphi^p(x)$.
\item For any $x\in \B_0$, $x$ is $\varphi$-periodic or there is $n\in \NN$ such that  $$c\left(\varphi^n(x),\varphi^{2n}(x),\varphi^{3n}(x)\right)\in\left]c\left(x,\varphi^n(x),\varphi^{2n}(x)\right),c\left(\varphi^{2n}(x),\varphi^{3n}(x),\varphi^{4n}(x)\right)\right[.$$
\item For any $x\in \B$ such that $x$ is not $\varphi$-periodic, there are $y,z\in \B$ $\varphi$-periodic points such that  $x\in D(y,z)$.
\item For any $x,y\in\B_0$ such that there are $n,m\in\NN$ with $\varphi^n(x)\in]x,\varphi^{2n}(x)[$ and $\varphi^m(y)\in]y,\varphi^{2m}(y)[$. 
\begin{itemize}
\item if the $\varphi^n$-orbit of $x$ and the $\varphi^m$-orbit of $y$ are separated (i.e. no point of one of the orbit is between two points of the other) then there is $z\in\B$, $\varphi$-periodic point such that $z\in [x,y]$,
 \item in the other case there are $x_0,y_0\in \B_0$ and $k\in\NN$ such that 
 \begin{itemize}
 \item the $\varphi$-orbit of $x$ is $\{x_0,\dots,\varphi^{k}(x_0)\}$, the $\varphi$-orbit of $y$ is $\{y_0,\dots,\varphi^{k}(y_0)\}$, 
 \item $y_0\in[x_0,\varphi^l(x_0)]$ or $x_0\in[y_0,\varphi^l(y_0)]$ where $l$ is the minimum such that  $\varphi^l(x_0)\in]x_0,\varphi^{2l}(x_0)[$ or $\varphi^l(y_0)\in]y_0,\varphi^{2l}(y_0)[$ 
\item and $k$ is a multiple of $l$.
\end{itemize}
 \end{itemize}
 \item If $x\in\B$ and $y,z\in \B$ are $\varphi$-periodic points such that $x\in D(y,z)$ then the length of the $\varphi$-orbits of $x$ and of $c(x,y,z)$ are the same.
  \item The set $\A$ is the $c$-closure of $\B$ and $\C$. That is, for any $x\in \A$, there are $x_1,x_2,x_3\in \B\cup\C$ such that $x=c(x_1,x_2,x_3)$.  
  \end{enumerate}\end{defn}
 
 Let us explain a bit this definition. The first point means that there is some initial set $\B_0$ such that any point of $\B$ lies in some positive $\varphi$-orbit of $\B_0$. For the second point, it means that any point of $\B_0$ is a fixed point of $g^n$ or lies in $D(g^n)$ for any $g\in G_\infty$ that induces $\varphi$ (Lemma~\ref{lem:c}). This removes the indetermination that appears in Remark~\ref{notAP}. More precisely, these conditions imply that no extension of such a system can merge distinct $\varphi$-orbits (Lemma~\ref{atmostoneorbit}). The third point means that if $x\in D(g^n)$ where $g$ induces $\varphi$ then it lies in the connected component between two fixed points of $g^n$. Condition (4) means that if $x,y$ lie in some austro-boreal part  for some power of $g$ inducing $\varphi$ and the orbits under these powers do not intertwine then they are separated by some periodic point.

\begin{lem}\label{fixpt} The class of systems  $\mathcal{S}=\langle \mathbf{A},\varphi\colon\mathbf{B}\to\mathbf{C}\rangle\in\mathcal{K}_p$ such that there is $x_0\in \mathbf{B}$ with $\varphi(x_0)=x_0$ is cofinal in $\mathcal{K}_p$.
\end{lem}
\begin{proof}Since $\Br(D_\infty)$ is a Fraïssé limit, we know that for any system $\mathcal{S}=\langle \mathbf{A},\varphi\colon\mathbf{B}\to\mathbf{C}\rangle$, we may identify $\mathbf{A}$ with a subset of $\Br(D_\infty)$ and $\varphi$ with some restriction to $\mathbf{A}$ of an automorphism $g\in\Homeo(D_\infty)$. Since dendrites have the fixed point property, $g$ has a fixed point $x$ in $D_\infty$. If this point $x$ is a branch point, it suffices to add it to $\mathbf{A},\mathbf{B}$ and $\mathbf{C}$ (and possibly the finitely many points $c(x,y,z)$ for $y,z\in \mathbf{A},\mathbf{B}$ or $\mathbf{C}$) and to extend $\varphi$ with $\varphi(x)=x$ (or by the value of $g$ on the points $c(x,y,z)$ for $y,z\in\mathbf{B}$). If $x$ is not a branch point then we can reduce to the situation where $x$ is a fixed branch point by the following construction. We glue infinitely many copies of $D_\infty$ to $x$ by identifying a branch point of each copy with $x$. We extend $g$ on the new branches around $x$ by the identity (which is possible thanks to the patchwork lemma). The new dendrite is homeomorphic to $D_\infty$ once again and we are back in the situation where $g$ has a fixed branch point.
\end{proof}

Let  $\mathcal{S}=\langle \mathbf{A},\varphi\colon\mathbf{B}\to\mathbf{C}\rangle\in\mathcal{K}_p$ be a system with a fixed point $x_0\in\mathbf{B}$. We define a \emph{branch} around $x_0$ to be an equivalence class in $\mathbf{A}\setminus\{x_0\}$ under  the relation $x\sim_{x_0} y \iff  \neg B(x_0;x,y)$.
For two branches around $x_0$, we write $D_1\sim_\varphi D_2$ if there is $x\in D_1\cap \mathbf{B}$ such that $\varphi(x)\in  D_2$. Observe that for another $y\in \mathbf{B}\cap D_1$ then $\varphi(y)\in D_2$ because $\varphi$ preserves the betweenness relation. In that case, we write $\varphi(D_1)=D_2$ even if this equality is not true for the underlying sets (we only have $\varphi(D_1\cap \mathbf{B})\subset D_2$). We also take the liberty to write recursively $\varphi^n(D_1)$ for $\varphi(\varphi^{n-1}(D_1))$ if $\varphi^{n-1}{(D_1)}\cap \mathbf{B}\neq\emptyset$. We still denote by $\sim_\varphi$ the equivalence relation generated by this relation. A $\varphi$-\emph{orbit of branches} is an equivalence class of branches under this equivalence relation. 

\begin{lem}\label{phi-orbit} For any $\varphi$-orbit of branches $E$ around $x_0$, there is a branch $D$ and $n\in\NN$ with $D\cap B\neq\emptyset$ such that $E=\{D,\varphi(D),\dots,\varphi^{n-1}(D)\}$.
\end{lem}
\begin{proof} Let us first prove that if $D$ and $D’$ are two branches around $x_0$ such that $\varphi(D)=\varphi(D’)$ then $D=D’$. In fact, if $x\in D$, $y\in D’$ with $\neg B(x_0;\varphi(x),\varphi(y))$ then $\neg B(x_0;x,y)$ and thus $D=D’$.

If $D,D’$ are in $E$ then there is chain $D_0,\dots,D_k$ such that $D_0=D$, $D_k=D’$ and $\varphi(D_i)=D_{i+1}$ or $\varphi(D_{i+1})=D_i$ for each $i=0,\dots,k-1$. One shows by induction on the length of the chain that $D=\varphi^k(D’)$ or $D’=\varphi^k(D)$.

Now, let $\{D,\varphi(D),\dots,\varphi^{n-1}(D)\}$ be a maximal such chain with distinct elements (which exists since $E$ is finite). Let $D’\in E$ then there is a minimal $k\in\NN$ such that $\varphi^k(D)=D’$ or $\varphi^k(D’)=D$. In the first case, by maximality of $\{D,\varphi(D),\dots,\varphi^{n-1}(D)\}$, $k\leq n-1$ and $D’\in\{D,\varphi(D),\dots,\varphi^{n-1}(D)\}$. In the second case, by maximality again, one has $\{D’,\varphi(D’),\dots,\varphi^{n-1+k}(D’)\}=\{D,\dots,\varphi^{n-1}(D)\}$. Thus $E=\{D,\varphi(D),\dots,\varphi^{n-1}(D)\}$.
\end{proof}

Observe that in Lemma~\ref{phi-orbit}, it is possible that $\varphi^{n-1}(D)\cap B\neq\emptyset$ and $\varphi^n(D)=D$.
Let $E$ be a $\varphi$-orbit of branches and $\mathbf{E}$ the union of its branches (which is $c$-closed). We define $\mathcal{S}_E=\langle \mathbf{A}_E,\varphi_E\colon \mathbf{B}_E\to \mathbf{C}_E\rangle$ where $\mathbf{A}_E=(\mathbf{A}\cap \mathbf{E})\cup\{x_0\}$, $\mathbf{B}_E=\mathbf{B}\cap\mathbf{A}_E$, $\mathbf{C}_E=\mathbf{C}\cap\mathbf{A}_E$ and $\varphi_E$ is the restriction of $\varphi$ to $\mathbf{B}_E$. Observe 
that $\mathcal{S}_E\in\mathcal{K}_p$ and if $\mathcal{S}\in\mathcal{L}$ then  $\mathcal{S}\in\mathcal{L}$ as well.

\begin{lem}\label{lem:addfixedpoint} Let $\mathcal{S}=\langle \mathbf{A},\varphi\colon \mathbf{B}\to\mathbf{C}\rangle\in \mathcal{K}_p$ be a system with a fixed point $x_0$ and point $x_1,\dots,x_k\in\mathbf{A}$ such that $x_1,\dots,x_{k-1}\in\B$, $x_1\notin \C$, $[x_{i+1},x_0]\subset [x_i,x_0]$ and $\varphi(x_i)=x_{i+1}$ pour all $i\leq k-1$. Let $\B_1=\{x\in \B\setminus\{x_1\},\ x_1\in[x_0,x]\ \textrm{and}\ x_1\in[x_0,\varphi(x)]\}$. Then there exists an extension $\mathcal{S}’=\langle \mathbf{A}’,\varphi’\colon \mathbf{B}’\to\mathbf{C}’\rangle$ of $\mathcal{S}$ such that  $\mathbf{A}’=\A\cup\{y\}$, $\mathbf{B}’=\B\cup\{y\}$, $\mathbf{C}’=\C\cup\{y\}$ and $\varphi’(y)=y$. Moreover, for any  $b\in \B_1, b’\in \B\setminus \B_1$, $y\in [b,b’]$.
\end{lem}

\begin{proof} Let us identify $\A$ with a subset of $\Br(D_\infty)$ and let $g\in G_\infty$ such that $\varphi$ is the restriction of $g$ on $\B$. Par construction, $x_1$ belongs to the interior of an austro-boreal arc $I$ of $g$. This arc is contained in the union of two connected components of $D_\infty\setminus\{x_1\}$. Let $U$ be the one that does not contain $x_0$. By definition of $\B_1$, $\B_1\subset U$. Let choose a branch point $y$ in the interior of  $I$ and in $U$ such that for any $b\in \B_1$, $y\in [x_1,b]$ and $gy\neq x_1\in [x_1,y]$. By construction, no element of $\B$ lies in $D(y,g(y))$. Let choose a slightly larger arc $[z,z’]$ in the interior of $I$ containing $[gy,y]$ and such that the preimage in $\B\cup\C$ of $[z,z’]$ by the first point map to $I$ is empty. Let $h$ be a homeomorphism of $D(z’,z)$ such that $h(g(y))=y$ and fixing $z,z’$. Let us extend $h$ to an element of $G_\infty$ by setting $h$ to be trivial outside $D(z,z’)$. Now let $g’=h\circ g$ and $\varphi’$ to be its restriction on $\B’=\B\cup\{y\}$.
\end{proof} 

\begin{prop}\label{prop:cofinal}The class $\mathcal{L}$ is cofinal in $\mathcal{K}_p$.
\end{prop}
\begin{proof}Let $\mathcal{S}=\langle \mathbf{A},\varphi\colon\mathbf{B}\to\mathbf{C}\rangle\in\mathcal{K}_p$. Thanks to Lemma~\ref{fixpt}, we may assume that $\varphi$ has a fixed point $x_0$.  
Moreover, if $g\in G_\infty$ induces $\varphi$, one can replace $\B$ by the $c$-closure of $\B\cup \left(\A\setminus \B\cup\C\right)$, $\C$ by the image of this new $\B$ by $g$ and $\A$ the $c$-closure of $\B\cup \C$. Thanks to the patchwork lemma, we may reduce to the case where $\varphi$ has a unique orbit of branches $E$ around $x_0$ and thus $\mathcal{S}=\mathcal{S}_E$. Actually, we can deal with each orbit of branches separately and patch them together at the end. So let us assume that $\mathcal{S}=\mathcal{S}_E$ and let $D$ be given by Lemma~\ref{phi-orbit}. 

The proof will be done thanks to different reductions and inductions.\\

\textbf{Case A.} The orbit of branches $E$ is reduced to $D$, that is $n=1$ in Lemma~\ref{phi-orbit}. On $\mathbf{A}\setminus\{x_0\}$, we define a partial order $x\leq y \iff x\in[x_0,y]$. This is a semi-linear order (see \cite[\S 5]{DM_dendritesII}) and thus, since $\mathbf{B}’=\mathbf{B}\setminus\{x_0\}$ and $\mathbf{C}’=\mathbf{C}\setminus\{x_0\}$ are $c$-closed, they have a unique minimum that we denote respectively by $b_0$ and $c_0$. Since $\varphi$ preserves betweenness, we know that $\varphi(b_0)=c_0$. We set $\mathbf{F}=\mathbf{B}\cup\mathbf{C}\cup\{c(x_0,b_0,c_0)\}$.

\textit{Subcase A.1} The point $c(x_0,b_0,c_0)\notin\{b_0,c_0\}$. In this case, no point of $\mathbf{B}’$ is between points of $\mathbf{C}’$ and vice versa. In particular, they are disjoint. We can define an extension $\langle \mathbf{A},\psi\colon  \mathbf{F}\to\mathbf{F}\rangle$ where $\psi|_\mathbf{B}=\varphi$, $\psi(c(x_0,b_0,c_0))=c(x_0,b_0,c_0)$ and for $c\in\mathbf{C}$, $\psi(c)=b\in\mathbf{B}$ is the unique point $b\in\mathbf{B}$ such that $\varphi(b)=c$. This extension belongs to $\mathcal{K}_p$ and it satisfies  Definition~\ref{class} with $\mathbf{F}_0=\mathbf{B}\cup\{c(x_0,b_0,c_0)\}$. Actually, for any $b\in\mathbf{F}_0$, $\psi^2(b)=b$ and thus Condition (2) is satisfied. Conditions (3) \& (4) are empty and $\A$ is still the $c$-closure of $\mathbf{F}$.

\textit{Subcase A.2} The point $c(x_0,b_0,c_0)$ is $c_0$. Observe that for any $g\in G_\infty$ that induces $\varphi$, $b_0$ belongs to an austro-boreal arc for $g$ because $[b_0,x_0]$ is mapped to $[c_0,x_0]$ and thus $c_0=g(b_0)$ belongs to $[b_0,g^2(b_0)]$. Let us denote by $x_1,\dots, x_k$ the $\varphi$-orbit of $b_0$ such that $\varphi(x_i)=x_{i+1}$ (in particular $b_0=x_{k-1}$ and $c_0=x_k$). Let $\mathbf{B}_1=\{x\in \mathbf{B},\ x_1\notin]x,\varphi(x)[\ \textrm{and}\  x_1\in]x,x_0[\}$.  

Thanks to Lemma~\ref{lem:addfixedpoint}, we may assume that $\varphi$ has a fixed point $y\in\mathbf{B}_1$ between $x_1$ and any other point of $\B\setminus\mathbf{B}_1$. Let $\mathbf{C}_1$ to be $\varphi(\mathbf{B}_1)$ and $\A_1$ to be $\{y\}$ and the union of the branches in $\A$ around $y$ that do not contain $b_0$. In particular, $\A_1$ is $c$-closed, contains  $\B_1\cup\C_1$. We define $\mathcal{S}_1=\langle\A_1, \varphi_1\colon \B_1\to\C_1\rangle$ where $\varphi_1$ is the restriction of $\varphi$ to $\B_1$. By an induction on the number of $\varphi_1$-orbits which is less than the number of $\varphi$-orbits, we may assume that $\mathcal{S}_1$ embeds in $\mathcal{S}_1’\in\mathcal{L}$.

Let us define $\A_2=(\A\setminus \A_1)\cup\{y\}$, $\B_2=\B\cap \A_2$, $\C_2=\C\cap\A_2$ and let $\varphi_2$ be the restriction of $\varphi$ on $\B_2$. Let choose $g\in G_\infty$ that induces $\varphi$. We set $\B_2’$ to be $\B_2$ and we add successively points to this set. For any point $x$ in $\B_2\cap]y,x_0[$, we know that $x\geq b_0$ and $g(x)\leq x_1$. We add to $\B_2’$ all points $g^m(x),g^{m+1}(x),\dots,g^l(x)$ such $m$ is the maximal integer such that $g^m(x)\geq x_1$ and $l$ is the minimal integer such that $g^l(x)\geq b_0$. For $z\in \B_2$ not in $[y,x_0]$, let $x_z\in]y,x_0[$ be $c(z,x_0,y)\in \B_2$. Let $m,l$ be the corresponding integers for $x_z$, we add $\{g^m{z},\dots,g^l{z}\}$ to $\B_2’$ in order to guarantee Condition (5). Finally, we  replace $\B’_2$ by its $c$-closure (this adds only finitely many points). Let $\C_2’=g(\B_2’)$. Let us define $(\B_2’)_0=\left(\B_2’\cap \overline{D(y,x_1)}\right)\cup\{x_0\}$. Up to add $g^i((\B_2’)_0)$ to $\B’_2$ for $i=1,2,3,4$ (and $g^i((\B_2’)_0)$ to $\C’_2$ for $i=2,3,4,5$) we may assume that $l-m>4$. Now, let $\A_2’$ be the $c$-closure of $\A_2\cup\B_2’\cup\C_2’$ and $\varphi’_2$ be the $g|_{\B_2’}$ . The system $\langle \A_2, \varphi_2\colon\B_2\to\C_2\rangle$ embeds in $\mathcal{S}_2’=\langle \A_2’, \varphi_2’\colon\B_2’\to\C_2’\rangle\in\mathcal{K}_p$. Moreover, the latter one satisfies Condition (2) in Definition \ref{class} with $n=1$ for all points. Actually Condition (1) is obtained by construction of $(\B_2’)_0$, $x_0$ and $y$ are fixed points and all the other points of $(\B_2’)_0$ are in $D(g)$ thus Conditions (2) \& (3) follow. For Condition (4), any two points of $\B’$ that lie in $]y,x_0[$ have intertwined $\varphi$-orbit and the last possibility of Condition (4) occurs. So $\mathcal{S}'_2\in\mathcal{L}$.

By the patchwork lemma, there is $g\in D_\infty$ inducing $\varphi_1'$ and $\varphi_2'$. Thus if $\varphi'$ is the restriction of $g$ on  $\B_1\cup\B_2$ $\mathcal{S}'=\langle \A_1\cup\A_2',\ \varphi'\colon \B_1\cup\B_2'\to\C_1\cup\C_2'\rangle\in\mathcal{L}$ is an extension of $\mathcal{S}$.

\textit{Subcase A.3} The point $c(x_0,b_0,c_0)$ is $b_0$. This subcase is very similar to subcase A.2 and we only indicate what should be modified. The points $x_1,\dots x_k$ are the $\varphi$-orbit of $b_0$ but this time $x_1=b_0$ and $x_{i+1}=\varphi(x_i)$.  Let $\mathbf{B}_1=\{x\in \mathbf{B},\ x_k\notin]x,\varphi(x)[ \ \textrm{and}\  x_k\in ]x_0,x[\}$.  

Thanks to Lemma~\ref{lem:addfixedpoint} applied to $\langle \A,\varphi^{-1}\colon \C\to\B\rangle$, we may assume that $\varphi$ has a fixed point $y\in\mathbf{B}_1$ between $x_k$ and any other point of $\B\setminus\mathbf{B}_1$. Let $\mathbf{C}_1$ to be $\varphi(\mathbf{B}_1)$ and $\A_1$ to be $\{y\}$ and the union of the branches in $\A$ around $y$ that do not contain $b_0$. In particular, $\A_1$ is $c$-closed and  contains  $\B_1\cup\C_1$. We define $\mathcal{S}_1=\langle\A_1, \varphi_1\colon \B_1\to\C_1\rangle$ where $\varphi_1$ is the restriction of $\varphi$ to $\B_1$. By an induction on the number of $\varphi_1$-orbits, we may assume that $\mathcal{S}_1$ embeds in some $\mathcal{S}_1’\in\mathcal{L}$ where $\varphi_1$ is the restriction of $\varphi$ on $\B_1$.

Let us define $\A_2=\A\setminus \A_1\cup\{y\}$, $\B_2=\B\cap \A_2$, $\C_2=\C\cap\A_2$ and let $\varphi_2$ be the restriction of $\varphi$ on $\B_2$. Let choose $g\in G_\infty$ that induces $\varphi$. We set $\B_2’$ to be $\B_2$ and we add successively points to this set. For any point $x$ in $\B_2\cap]y,x_0[$, we know that $x\geq b_0$ and $x\leq x_k$. We add to $\B_2’$ all points $g^m(x),g^{m+1}(x),\dots,g^l(x)$ such that $l$ is the minimal integer such that $g^m(x)\geq x_k$ and $l$ is the minimal integer such that $g^l(x)\geq b_0$. For $z\in \B_2$ not in $[y,x_0]$, let $x_z\in]y,x_0[$ be $c(z,x_0,y)\in \B_2$. Let $m,l$ be the corresponding integers for $x_z$, we add $\{g^m{z},\dots,g^l{z}\}$ to $\B_2’$. Finally, we  replace $\B’_2$ by its $c$-closure (this adds only finitely many points). Let $\C_2’=g(\B_2’)$. Let us define $(\B_2’)_0=\left(\B_2’\cap \overline{D(x_1,x_2)}\right)\cup\{x_0\}$. Up to add $g^i((\B_2’)_0)$ to $\B’_2$ for $i=1,2,3,4$ (and $g^i((\B_2’)_0)$ to $\C’_2$ for $i=2,3,4,5$) we may assume that $l-m>4$. Now, let $\A_2’$ be the $c$-closure of $\A_2\cup\B_2’\cup\C_2’$ and $\varphi’_2$ be the $g|_{\B_2’}$ . The system $\langle \A_2, \varphi_2\colon\B_2\to\C_2\rangle$ embeds in $\mathcal{S}_2’=\langle \A_2’, \varphi_2’\colon\B_2’\to\C_2’\rangle\in\mathcal{K}_p$ and $\mathcal{S}_2’\in\mathcal{L}$ for the same reasons as above. We conclude similarly as in Subcase A.2.
\bigskip

\textbf{Case B.} The integer $n$ is larger than 1. The idea is then to reduce to Case A by making a precise definition for $\psi=\varphi^n$ and apply Case A to $\psi$.

\textit{Subcase B.1} $\varphi^{n-1}(D)\cap \mathbf{B}=\emptyset$. This case is quite similar to subcase A.2.
Let us identify $\mathbf{A}$ with a subset of $\Br(D_\infty)$ and $\varphi$ with the restriction of some $g\in\Homeo(D_\infty)$. Let $\widetilde{D}$ be the connected component of $D_\infty\setminus\{x_0\}$ that contains the points of $D$. Let $\B_0$ be the $c$-closure of $\widetilde{D}\cap\left(\bigcup_{0\leq k< n-1}g^{-k}(\B)\right)\cup\{x_0\}$. This is a finite set. Let $\mathbf{C}’=\mathbf{B}’=\bigcup_{0\leq k< n-1}g^k(\B_0)$. We define $\varphi’$ to coincide with $g$ on $\bigcup_{0\leq k< n-1} g^k(\B_0)$ and $g^{-(n-1)}$ on $g^{n-1}(\B_0)$. The class $\mathcal{S}’=\langle \mathbf{A}’,\varphi\colon\mathbf{B}’\to\mathbf{C}’\rangle$ which is an extension of $\mathcal{S}$ belongs to $\mathcal{L}$ with this integer $n$ and all points  of $\B_0$ are $\varphi^n$-fixed points.

\textit{Subcase B.2} $\varphi^n(D)\cap \mathbf{B}\neq\emptyset$. So $\varphi^n(D)=D$.  Up to choose $g\in G_\infty$ and add points in each $\varphi$-orbits, we may assume that all $\varphi$-orbits start, finish in $D$ and have  all the same length that is at least $4n$. That is,  we may assume that $\B_0=D\cap\mathbf{B}$ satisfies $\bigcup_{0\leq k \leq ln-1}\varphi^k(\B_0)=\mathbf{B}$ with $l\geq4$ and $\B_0$ is $c$-closed. This will guarantee the very last point of Condition (4) in Definition~\ref{class}. We define the system $\mathcal{T}=\langle \A\cap D\cup\{x_0\},\psi\colon\bigcup_{0\leq k \leq l-1}\varphi^{kn}(\B_0)\to\bigcup_{0\leq k \leq l-1}\varphi^{(k+1)n}(\B_0)\rangle$ where $\psi$ is the restriction of $\varphi^n$ to $\bigcup_{0\leq k \leq l-1}\varphi^{kn}(\B_0)$. Observe that the number of $\psi$-orbits is at most the same number of $\varphi$-orbits. Now, $\mathcal{T}\in \mathcal{K}_p$ and falls in Case A. So we can find an embedding of $\mathcal{T}$ in some $\mathcal{T}’=\langle \A’,\psi’\colon \B’\to\C’\rangle\in\mathcal{L}$. Let us define $\B"=\bigcup_{0\leq i<n}g^i(\B’)$ and $\varphi’$ to be $g$ on $\bigcup_{0\leq i<n-1}g^i(\B’)$ and $\psi’\circ g^{-(n-1)}$ on $g^{n-1}(\B’)$. Let $\C"=\varphi’(\B")$ and $\A"=\A’\cup\B"\cup\A$. The system $\mathcal{S}’=\langle\A",\varphi’\colon \B"\to\C"\rangle$ lies in $\mathcal{K}_p$ and contains an embedding of $\mathcal{S}$. Moreover, $\mathcal{S}’\in\mathcal{L}$ because $\mathcal{T}’\in\mathcal{L}$. Actually if $\B_0$ is the initial set for $\mathcal{T}’$ as in Definition~\ref{class} then it is also an initial set for $\mathcal{S}’$ and for $x\in\B_0$ and $n’\in\NN$ is as in Definition~\ref{class}.(2) for $\mathcal{T’}$ then $nn’$ satisfies Conditions (2)-(4) for $x$ and $\varphi’$. Condition (5) is satisfied since it is satisfied for $\mathcal{T’}$.

Before the subdivision into cases A and B, Condition (6) was guaranteed and during the extension that we did in cases A and B, no point outside the $c$-closure of $\B\cup\C$ was added to $\A$ and thus Condition (6) is still satisfied at the end.
\end{proof}

\begin{prop}\label{WAP} The class $\mathcal{L}$ has the amalgamation property.\end{prop}

To prove this proposition, we rely on a few lemmas.

\begin{lem}\label{period} Let $\mathcal{S}=\langle \A,\varphi\colon\B\to\C\rangle\in\mathcal{L}$ and $\mathcal{S}\to \mathcal{T}=\langle \mathbf{D},\psi\colon \mathbf{E}\to\mathbf{F}\rangle$ be an embedding. If $x,y\in\B$ are $\varphi$-periodic points with periods $n\leq m$. Assume that $]x,y[\cap\B$ does not contain any periodic point then the components $g^k(D(x,y))$ are disjoint for $0\leq k<m$ and for any $z\in \mathbf{D}$, $\psi$-periodic point with $z\in]x,y[$, the period of $z$ is $m$.
\end{lem}

\begin{proof} Let $g\in G_\infty$ inducing $\psi$. We claim that for all $g^k(D(x,y))$ are disjoint for $0\leq k<m$ and $m$ is a multiple of $n$. This implies that the period of $z$ is at least $m$. Now, we have $g^m(]x,y[)=]x,y[$. If the period of $z$ is not $m$ then $g^m(z)\in]z,g^{2m}(z)[$ and $z\in D(g^m)$. Thanks to Lemma \ref{lem:c} and \ref{KD}, this leads to a contradiction. 

It remains to prove the claim.  If $m=1$, the claim is straightforward. So let us assume that $m>1$. Since $\mathcal{S}\in\mathcal{L}$, $g$ has some fixed point $p\in\B$. Since $m>1$, $y$ is not in the element of $\mathcal{C}_x$ that contains $p$. So, if $k$ is not a multiple of $n$, then $c(p,x,g^k(x))$ separates $D(x,y)$ and $g^k(D(x,y))$. Thus $D(x,y)$ and $g^k(D(x,y))$ are disjoint. Now, if $D(x,y)$ and $g^k(D(x,y))$ are not disjoint then the point $c(x,y,g^k(y))$ is necessarily a non-periodic point because otherwise it would belongs to $\B$ which is $c$-closed. By assumption there is no such point. 
\end{proof}

\begin{lem}\label{lem:fix}Let $g\in G_\infty$, let  $x,y$ be $g$-fixed points in $D_\infty$ and let $M$ be some finite set in $]x,y[$. There are $z\in ]x,y[$ such that $M\subset]x,z[$, $p\in]z,y[$ and $g'\in G_\infty$ that is equal to $g$ on $D_\infty\setminus D(z,y)$ and that fixes $p$.
\end{lem}

\begin{proof}Since $M$ is finite, one can find $z\in ]x,y[$ such that $M\subset]x,z[$. Now, choose $p\in]z,y[\cap]gz,y[\subset]x,y[$. Find a homeomorphism $f$ from $D(z,y)$ to $D(gz,y)$ fixing $p,y$ and such that $f(z)=g(z)$ (this is possible thanks to \cite[Proposition 6.1]{DM_dendritesII}). Now define, $g'$ to be $f$ on $D(z,y)$ and $g$ elsewhere.
\end{proof}

\begin{lem}\label{lem:per} Let $\mathcal{S}=\langle \A,\varphi\colon\B\to\C\rangle\in\mathcal{L}$ and $\mathcal{S}\to \mathcal{T}=\langle \mathbf{D},\psi\colon \mathbf{E}\to\mathbf{F}\rangle$ be an embedding. Let $x,y\in\B$ be $\varphi$-periodic points. Assume that $]x,y[\cap\B$ does not contain any periodic point. Then there is an embedding $\mathcal{T}\to \mathcal{T}'=\langle \mathbf{D}’,\psi’\colon \mathbf{E}’\to\mathbf{F}’\rangle$ such that there is $p\in \mathbf{D}’$ with $\mathbf{D}’=\mathbf{D}\cup\{p\}$, $\mathbf{E}’=\mathbf{E}\cup\{p\}$, $\mathbf{F}’=\mathbf{F}\cup\{p\}$, $p\in ]x,y[$, $z$ is $\psi’$-periodic and $]p,y[\cap\mathbf{D}=\emptyset$.
\end{lem}

\begin{proof} Let $h\in G_\infty$ that induces $\psi$ and let us consider $\mathbf{D}$ as a subset of $\Br(D_\infty)$. As in the proof of Lemma \ref{period}, let $m$ be the maximal period of $x$ and $y$. So, the subsets $h^k(D(x,y))$ are disjoint for $0 \leq k<m$. Let us set $g=h^m$. So, $x$ and $y$ are $g$-fixed points. Let us apply Lemma \ref{lem:fix} with $g=h^m$ and $M=]x,y[\cap \mathbf{D}$. One get $g’$ that fixes some point $p$ such that $D(p,y)\cap \mathbf{D}=\emptyset$ and $g$ coincides with $g$ on $\mathbf{E}$. Now, thanks to the patchwork lemma, let us define $h’\in G_\infty$ to be $h$ on $D_\infty\setminus h^{m-1}(D(x,y))$ and $g’\circ h^{1-m}$ on $D(x,y)$. Then $h’$ coincides with $h$ on $\mathbf{E}$ and $p$ is $h’$-periodic. We define $\psi’$ to be the restriction of $h’$ on $\mathbf{E}\cup\{p\}$.
\end{proof}

\begin{lem}\label{atmostoneorbit}Let $\mathcal{S}=\langle \A,\varphi\colon\B\to\C\rangle\in\mathcal{L}$ and $\mathcal{S}\to \mathcal{T}=\langle \mathbf{D},\psi\colon \mathbf{E}\to\mathbf{F}\rangle$ be an embedding. Any $\psi$-orbit contains at most one $\varphi$-orbit.\end{lem}

\begin{proof} For a contradiction, let us assume there are $x,y\in\B$ such that their $\varphi$-orbits are distinct but lie in the same $\psi$-orbit. Thanks to Condition (1) in Definition~\ref{class}, we may assume that $x,y\in \B_0$. Observe that these points are not $\varphi$-periodic (and thus not $\psi$) because the $\psi$-orbit of a periodic point is equal to its $\varphi$-orbit .

The points $x,y$ belong respectively to some $D(x_1,x_2),D(y_1,y_2)$ where $x_1,x_2,y_1,y_2$ are $\varphi$-periodic points and we may assume that $D(x_1,x_2),D(y_1,y_2)$ do not contain $\varphi$-periodic points. Thus, the $\varphi$-orbits of $D(x_1,x_2),D(y_1,y_2)$ are disjoint or the equal and this is a fortiori the same for $\psi$. So, under our current assumption, these two orbits are the same. Let $x’=c(x,x_1,x_2)$ and $y’=c(y,y_1,y_2)$. Because of Condition (5), $x’, y’\in\B_0$. By Condition (2) of Definition~\ref{class}, there is $n$ (that we can assume to be minimal) such that for any $g\in G_\infty$ that induces $\varphi$, $g^n$ is austro-boreal on $D(x_1,x_2)\cap\B$ and on $D(y_1,y_2)\cap\B$. So $x’,z’$ satisfy Condition (4) and their $\varphi^n$-orbits are not separated by $\varphi^n$-fixed point, so by  the second point of Condition (4), $x’\in[y’,\varphi^{n}(y’)]$ or $y’\in[x’,\varphi^{n}(x’)]$. Since $\psi$ commute with the center map, $x'$ and $y'$ are in the same $\psi^n$-orbit. This is impossible because the $\psi^n$-iterates of $[y',\psi^n(y')[$ (respectively of $[x',\psi^n(x')[$) are distinct.
\end{proof}

\begin{proof}[Proof of Proposition \ref{WAP}] Let $\mathcal{S}=\langle \A,\varphi\colon\B\to\C\rangle\in\mathcal{L}$ and two embeddings $\iota_i\colon\mathcal{S}\to \mathcal{S}_i=\langle \A_i,\varphi_i\colon \B_i\to\C_i\rangle$ where $\mathcal{S}_i\in\mathcal{L}$ for $i=1,2$. We will construct two embeddings $j_i\colon \A_i\to \Br(D_\infty)$ and $g\in G_\infty$ such that $j_1\circ \iota_1=j_2\circ \iota_2$ and $j_i\circ\varphi_i=g\circ j_i$ for $i=1,2$. So, the restriction of $g$ on $j_1(\B_1)\cup j_2(\B_2)$ 
will yield an amalgamation of $\varphi_1$ and $\varphi_2$ over $\varphi$.

First, we fix, an embedding $j\colon \A\to\Br(D_\infty)$ and $g\in G_\infty$ such that $g$ induces $\varphi$ on $j(\B)$. For simplicity, we write $\A$ instead of  $j(\A)$ , so we think to $\A$ as a subset of $\Br(D_\infty)$. We define $j_i$ on $\iota_i(\A)$ to be $j\circ\iota_i^{-1}$ and it remains to define $j_i$ on $\A_i\setminus\iota_i(\A)$.

For a point $x\in \B_i\setminus\B$ there are three exclusive cases :

\begin{enumerate}[label=(\Alph*)]
\item The $\varphi_i$-orbit of $x$ contains a $\varphi$-orbit of a point in $\B$.
\item The $\varphi_i$-orbit of $x$ does not contains a $\varphi$-orbit of a point in $\B$ but there are points $y,z\in \B\cup \C$ such that $x\in D(y,z)$. 
\item The $\varphi_i$-orbit of $x$ does not contains a $\varphi$-orbit of a point in $\B$ and there are no points $y,z\in \B\cup \C$ such that $x\in D(y,z)$. 
\end{enumerate}

Let us observe that if  $\varphi_i$-orbit of $x$ contains a $\varphi$-orbit of a point $y\in\B$ then this point $y$ satisfies the second possibility in Condition (2) of Definition \ref{class}. That is, it lies in $D(h^n)$ for some $n\in\NN$ and any $h\in G_\infty$ that induces $\varphi_i$. Moreover thanks to Lemma \ref{atmostoneorbit}, in this situation, the $\varphi_i$-orbit of $x$ contains exactly one $\varphi$-orbit.


We first deal with points in cases (A)\&(B). These points in $\B_i\setminus\B$ lie in some $D(y,z)$ where $y,z$ are $\varphi$-periodic points and thanks to point (3) in Definition \ref{class}, there is $n\in\NN$ such that $\varphi^n(y)=y$ and $\varphi^n(z)=z$.

\textbf{Case A.} Let $\B_{0,i}$ the sets given by Definition \ref{class} for $\mathcal{S}_i$. Let  $x\in\B_{0,i}$ such that its $\varphi_i$-orbit contains the $\varphi$-orbit of some $y\in\B_0$. So, there is $k\in\NN$ such that $\varphi_i^k(x)=y$. We define $j_i(x)=g^{-k}(y)$. For $z\in\B_i$ in the $\varphi_i$-orbit, there is $l\in\NN$ such that $z=\varphi_i^l(x)$ and we define $j_i(z)=g^l(j_i(x))$. Since all points in $\B_i\setminus\B$ are in $\varphi_i$-orbit of some point in $\B_{0,i}$, we are done with points that fall in case (A).

\textbf{Case B.} We consider now points $x$ in $\B_i$ that are not in Case A but lie in some $D(y,z)$ for some $y,z$ $\varphi$-periodic points. Let us fix $y,z$ $\varphi$-periodic points such that $[y,z]$ does not contain any other $\varphi$-periodic point. In particular, any two points in $\B_0\cap]y,z[$ satisfy the second property of Condition (4) in Definition \ref{class}.

The components $D(g^i(y),g^i(z))$, for $i\in\ZZ$, are  disjoints or equal and because of Lemma~\ref{period}.  Among them, at most one meets $\B_0$ (and similarly for $B_{0,1}$ and $B_{0,2}$). If none meets $\B_0$ then  $D(g^k(y),g^k(z))\cap\B=\emptyset$ for any $k$. Because of Condition (2) in Definition \ref{class}, any $\varphi_i$-orbit that meet $D(y,z)\subset \B_i$ meets $D(\varphi^k_i(y),\varphi^k_i(z))$   as well. So, in case $D(y,z)\cap\B\neq\emptyset$, we may assume that $D(y,z)$ is the one that meets $\B_0$.  \\

Let us treat the case where $D(y,z)\cap\B=\emptyset$ first. In that case, we may assume there is a branch point $t\in ]y,z[$ that is $g$-periodic by Lemma~\ref{lem:per}. For $i=1,2$, we choose some $h_i\colon \A_i\to\Br(D_\infty)$ that coincides with $j$ on $\A$ and an element $g_i\in D_\infty$ that induces $\varphi_i$ on $\B_i$. Thanks to Lemma \ref{lem:per}, we may assume there are $t_i\in ]y,z[$ such that $h_1(\A_1)\cap D(y,z)\subset D(y,t_1)$,  $h_2(\A_2)\cap D(y,z)\subset D(t_2,z)$ and $t_i$ are $\varphi_i$-periodic points. By Lemma \ref{period}, the periods of $t,t_1$ and $t_2$ are the maximum of the periods of $y$ and $z$. Let $n$ be this period.

We may assume that $D(y,z)$ is the component that meets $B_{0,1}$ among all its $\varphi$-iterates. Let us fix a homeomorphism $l_1\colon D(y,t_1)\to D(y,t)$. For $k=0,\dots,n-1$, on $B_{1}\cap D(\varphi^k(y),\varphi^k(z))$, we define $j_1$ to be $g^k\circ l_1\circ h_1$. Finally, on $D(g^{k-1}(y),g^{k-1}(t))$, we replace the restriction of $g$ by $l_1\circ\varphi_1^n\circ l_1^{-1}\circ g^{1-n}$. This way, the embedding $j_1$ is well defined on $\cup_{0\leq k\leq n-1} D(\varphi^k(y),\varphi^k(z))\cap \B_1$ and we have $j_1\circ\varphi_1=g\circ j_1$ by construction.

We proceed similarly to embed points of $\cup_{0\leq k\leq n-1} D(\varphi^k(y),\varphi^k(z))\cap \B_2$ to $\Br(D_\infty)$. We fix a homeomorphism $l_2\colon D(t_2,z)\to D(t,z)$. For $k=0,\dots,n-1$, on $B_{2}\cap D(\varphi^k(y),\varphi^k(z))$, we define $j_2$ to be $g^k\circ l_2\circ h_2$. Finally, on $D(g^{k-1}(t),g^{k-1}(z))$, we replace the restriction of $g$ by $l_2\circ\varphi_2^n\circ l_2^{-1}\circ g^{1-n}$. This way, the embedding $j_2$ is well defined on $\cup_{0\leq k\leq n-1} D(\varphi^k(y),\varphi^k(z))\cap \B_2$ and we have $j_2\circ\varphi_2=g\circ j_2$ by construction.\\

Now, we assume that $D(y,z)\cap\B\neq\emptyset$ and continue with $n$ being the maximum of the period of $y$ and $z$. We may assume that $D(x,y)$ does not contain $\varphi$-periodic points, up to consider $D(y,z)$ minimal for inclusion with the property $D(y,z)\cap\B\neq\emptyset$.  For $x\in D(x,y)\cap\B_i$, we have three exclusive subcases that cover all possibilities.

\begin{enumerate}

\item The $\varphi_i$-orbit of $x$ contains a point in some $D(y’,z’)$ where $y’,z’\in \B\cap D(y,z)$ and $D(y’,z’)\cap \B=\emptyset$. 
\item The $\varphi_i$-orbit of $x$ does not contain a point in some $D(y’,z’)$ where $y’,z’\in \B\cap D(y,z)$ and $D(y’,z’)\cap \B=\emptyset$ but this $\varphi_i$-orbit contains a point in some $D(y’,z’)$ where $y’,z’\in (\B\cap D(y,z))\cup\{y,z\}$, $D(y’,z’)\cap \B=\emptyset$.
\item There is no point in the $\varphi_i$-orbit of $x$ that lies in some $D(y’,z’)$ where $y’,z’\in (\B\cap D(y,z))\cup\{y,z\}$.\end{enumerate}

\textit{Subcase B.1.} Let $y’,z’$ be such points. We will defined the embedding of the whole $\varphi_i$-orbit of $x$ and thus, we may assume that $x\in \B_{0,i}$. 
We continue with the embeddings $h_i$ defined above. We choose branch points $t,t_1,t_2\in ]y’,z’[$ such that $h_1(\A_1)\cap D(y’,z’)\subset D(y’,t_1)$,  $h_2(\A_2)\cap D(y’,z’)\subset D(t_2,z’)$. We fix homeomorphisms $l_1\colon D(y’,t_1)\to D(y’,t)$ and $l_2\colon D(t_2,z’)\to D(t,z’)$. We define $j_i$ to be $l_i\circ h_i$ on $D(y’,z’)\cap\B_i$. For elements $x\in \B_i$ such that $\varphi_i^k(x)\in D(y’,z’)$ for some $k\in\ZZ$, we define $j_i(x)$ to be $g^k\circ l_i\circ h_i \circ\varphi_i^{-k}(x)$. 
Thus we have defined $j_i$ for all elements whose $\varphi_i$-orbit meets $D(y’,z’)$ and for those points we have $j_i\circ\varphi_i=g\circ j_i$ by construction.

\textit{Subcase B.2.} In that case, $\{y,z\}\cap\{y’,z’\}$ is a point because $[y,z]$ contains points in $\B$. Let us assume, $y=y’$. The other possibilities are treated mutatis mutandis. Thanks to Lemma \ref{lem:per}, we may assume that there are $s,t\in ]y,z[$ such that $s\in ]y,t[$ and $s,t$ are $g$-periodic points. By Lemma \ref{period}, the periods of these points are necessarily  the maximum of the ones of $y$ and $z$, that is $n$. Moreover, $t$ is such that $\B\cap D(y,z)\subset D(t,z)$. Let us use the embeddings $h_i$ from above. From Condition (4) in Definition \ref{class}, we know that there is a $\varphi_i$-periodic point $t_i\in ]y,z[$ such that all points that fall in this second subcase with $y’=y$ have a $\varphi_i$-orbit that meets $D(y,t_i)$. 

We fix a homeomorphism $l_1\colon D(y, h_1(t_1))\to D(y,s)$ and define $j_1$ to be $l_1\circ h_1$ on $D(y,t_1)\cap\B_1$. 
For $k=0,\dots,n-1$, on $\B_{1}\cap D(\varphi_1^k(y),\varphi_1^k(t_1))$, we define $j_1$ to be $g^k\circ l_1\circ h_1$. Finally, on $D(g^{k-1}(y),g^{k-1}(s))$, we replace the restriction of $g$ by $l_1\circ\varphi_1^n\circ l_1^{-1}\circ g^{1-n}$. This way, the embedding $j_1$ is well defined on $\cup_{0\leq k\leq n-1} D(\varphi_1^k(y),\varphi_1^k(t_1))\cap \B_1$ and we have $j_1\circ\varphi_1=g\circ j_1$ on this subset by construction.

We may also assume that $\varphi_2$ has a periodic point $s_2\in]y,t_2[$ such that $D(y,s_2)\cap\B_2=\emptyset$. We fix a homeomorphism $l_2\colon D(h_2(s_2), h_2(t_2))\to D(s,t)$ and define $j_2$ to be $l_2\circ h_2$ on $D(s_2,t_2)\cap\B_2$. 
For $k=0,\dots,n-1$, on $\B_{2}\cap D(\varphi_2^k(s_2),\varphi_2^k(t_2))$, we define $j_2$ to be $g^k\circ l_2\circ h_2$. Finally, on $D(g^{k-1}(s),g^{k-1}(t))$, we replace the restriction of $g$ by $l_2\circ\varphi_2^n\circ l_2^{-1}\circ g^{1-n}$. This way, the embedding $j_2$ is well defined on $\cup_{0\leq k\leq n-1} D(\varphi_2^k(s_2),\varphi_2^k(t_2))\cap \B_2$ and we have $j_2\circ\varphi_2=g\circ j_2$ on this subset by construction.

\textit{Subcase B.3.} In this last subcase, for such an $x$, there is a unique point $p\in \B\cap D(y,z)$ such that for any $r\in \B$, $p\in[r,x[$. Moreover, this point is necessarily a $\varphi$-periodic point because of Condition (3) in Definition \ref{class}. Let $m$ be this period. Once again, we use the embeddings $h_i\colon \A_i\to \Br(D_\infty)$. Let $C_{i,1},\dots,C_{i,m_i}$, for $i=1,2$, be the connected components of $D_\infty\setminus \{p\}$ that contains points of $h_i(\B_i)$ but no point of $\B$ ($x$ is necessarily in such a component). Each of these components contain a $g_i$-periodic point. Thus for any $i=1,2$ and $j\leq m_i$, there is $k_{i,j}$ such that $g_i^{k_{i,j}}(C_{i,j})=C_{i,j}$.  

We glue copies of the $C_{i,j}$’s to $p$ and similarly we glue copies of $g_i^k(C_{i,j})$ to $g^k(p)$ for $k<m$ to obtain a new dendrite which again homeomorphic to $D_\infty$. We extend $g$ by the restriction of $g_i$ on the copies of $g_i^k(C_{i,j})$. Let define $j_i$ to be these gluings on $\cup_{0\leq k\leq n-1}g_i^k(C_{i,j})$. 

\textbf{Case C.} This last case is treated in the same way as subcase B.3 because for any point $x$ that fall in this case, there is a unique point $p$ such that for any $r\in \B$, $p\in[r,x[$ and this point $p$ is periodic.

To conclude this proof, we define $\B'$ to be the $c$-closure of $h_1(\B_1)\cup h_2(\B_2)$, $\C'$ to be $g(\B’)$, $\varphi'$ to be the restriction of $g$ on $\B'$ and $\A’$ to be the $c$-closure of $\B’\cup\C’$. 
\end{proof}

We can now conclude that $G_\infty$ has generic elements. 

\begin{proof}[Proof of Theorem \ref{ccc}] The class $\mathcal{K}_p$ has JEP (Proposition \ref{prop:jep}) and WAP  (Proposition \ref{prop:cofinal} and Proposition \ref{WAP}) thus  the theorem is a consequence of \cite{MR1162490} (see also \cite[\S3]{Kechris-Rosendal}).\end{proof}

Checking JEP and WAP conditions, we show similarly the existence of comeager conjugacy classes for the basic clopen subgroups.
\begin{thm}\label{cccv} Let $F\subset \Br(D_\infty)$ be a finite subset. The clopen subgroup $V_F=\Fix(F)$ has a comeager conjugacy class.\end{thm}

\begin{proof}We consider systems $\mathcal{S}=\langle \mathbf{A},\varphi\colon\mathbf{B}\to\mathbf{C}\rangle$ where $\varphi$ is induced by an element $g\in V_F$. The joint embedding and weak amalgamation properties are proved as for $G_\infty$.
\end{proof}

\begin{rem}Let us conclude this section by some observations. Let $A$ be the subgroup of $ G_\infty$ fixing a pair of points in $\Ends(D_\infty)$ and let $B$ be the stabilizer of some branch point in $D_\infty$. Notice that $V_F\cong A^n\times B^m$ where $n$ is the number of edges in $\langle F\rangle$ and $m$ is the number of vertices in $\langle F\rangle$ (see Section 2 for the definition of $\langle F\rangle$).

Once we identify the set of branch points in an open arc in $D_\infty$ with $\QQ$, we can also observe that $A$ is a permutational wreath product over $B$:
$$A\simeq B\wr_\mathbf{Q}\Aut(\QQ,<)\simeq B^\QQ\rtimes \Aut(\QQ,<).$$
The subgroup $B$ has itself a permutational wreath product decomposition where $E$ is the stabilizer of an end point in $D_\infty$:
$$B\simeq E^\NN\rtimes \sinf.$$
We refer to \cite[Lemma 7.1]{DM_dendritesII} for more explanations. Observe that intuitively this shows that $G_\infty$ is somehow built from the classical Polish groups $\sinf$,  $\Aut(\QQ,<)$ and $E$ which is the automorphism group of some semi-linear order on the set of branch points (\cite[Corollary 5.21]{DM_dendritesII}).\end{rem}

\section{Automatic continuity}\label{automatic}



Our proof of the automatic continuity relies on the Steinhaus property. To prove this property, we use the same technics as in the proof \cite[Theorem 15]{Rosendal-Solecki} which states that the Polish group $\Aut(\QQ,<)$ has the Steinhaus property. Let us recall that a topological group $G$ has the Steinhaus property (Definition~\ref{Stein}) if there is $k\in\NN$ such that for any symmetric and $\sigma$-syndetic subset $W$, $W^k$ contains a neighborhood of the identity. 

So, our goal is to prove that the Polish group $G_\infty$ has the Steinhaus property (Theorem~\ref{144-Steinhaus}). Before proving the theorem, let us set up a few things. Set merely $G=G_\infty$. Let $W$ be a symmetric $\sigma$-syndetic subset of $G_\infty$ (i.e. there is $(g_n)_{n\in\NN}$ with $\bigcup_{n\in\NN}g_nW=G$). Since $W$ is not meager, $W^2=W^{-1}W$ is dense in some open neighborhood of the identity $U=\Fix(F)$ where $F$ is a finite subset of  $\Br(D_\infty)$. Let us denote by $T=[F]$ the tree (i.e. subdendrite) generated by $F$ . 

Let $\mathcal{V}$ ($\supset F$) be the set of vertices of $T$ and $\mathcal{E}$ be its set of edges. For $v\in\mathcal{V}$, we denote  $G_v=\{g\in \Fix(v),\ \supp(g)\subset\cup \overline{U_i}\}$ where  $\{U_i\}$ are the connected components of $D_\infty\setminus\{x\}$ that do not intersect $T$. For $e=\{x,y\}\in\mathcal{E}$, we denote by $G_e=\{g\in G,\ \supp(g)\subset \overline{D(x,y)}\}$. Thanks to the patchwork lemma, we have

$$U=\left(\prod_{v\in\mathcal{V}}G_v\right)\times\left(\prod_{e\in\mathcal{E}}G_e\right).$$

If we set $V=\prod_{v\in\mathcal{V}}G_v$ and $E=\prod_{e\in\mathcal{E}}G_e$, that is $U=V\times E$, it suffices to show the following two lemmas to prove Theorem~\ref{144-Steinhaus}.

\begin{lem}\label{V}The subgroup $V$ is contained in $W^{140}$.
\end{lem}

\begin{lem}\label{E}The subgroup $E$ is contained in $W^{96}$.
\end{lem}

We can now conclude that $G_\infty$ has the Steinhauss property.
\begin{proof}[Proof of Theorem~\ref{144-Steinhaus}]
We have $U=V\times E\subset W^{140}\cdot W^{96}=W^{236}.$
\end{proof}

\begin{proof}[Proof of Lemma~\ref{V}] For $v\in\mathcal{V}$, let $\mathcal{C}_v^T$ be the set connected components of $D_\infty\setminus\{v\}$ that do no intersect $T$. We define a \emph{moiety} of $\bigcup_{v\in\mathcal{V}}\mathcal{C}_v^T$ to be a collection $X=(X_v)_{v\in \mathcal{V}}$ such that for each $v\in\mathcal{V}$,  $X_v$ is a moiety of $\mathcal{C}_v^T$, that is $X_v\subset \mathcal{C}_v^T$ is infinite and co-infinite. 
For such a moiety, we denote by $V(X)$ the subgroup of $V$ of elements supported on $\bigcup_{v\in \mathcal{V}}\bigcup_{C\in X_v}C\subset D_\infty$.

Let $(X_n)$ be a sequence of disjoint such moieties. For  each $n\in\NN$, let $g_n\in V(X_n)$. Thanks to the patchwork lemma, there is a well defined element $g\in V$ that coincides with each $g_n$ on its support. Let $(k_n)\in G^\NN$ such that $\bigcup_{n\in\NN}k_nW=G$. There is $n$ such that $V(X_n)$ is full for  some $k_nW$, that is for any $g\in V(X_n)$, there is $h\in k_nW$ such that $g$ and $h$ coincides on $X_n$. Otherwise, there would be $g_n\in V(X_n)$ such that no element of $k_nW$ coincides with $g_n$ on $X_n$. Thus the element $g$ obtained by patching the $g_n$’s would not be in $\bigcup_{n\in\NN}k_nW$. For the remaining of the proof, we fix $n$ such that $V(X_n)$ is full for  some $k_nW$. This implies that $V(X_n)$ is full for $W^2=(k_nW)^{-1}k_nW$ as well.

Observe that $V(X_n)\simeq B^\mathcal{V}$ where $B$ is stabilizer of a branch point in $G$. Since $B$ has a comeager conjugacy class (Theorem~\ref{cccv}), $V(X_n)$ has also a comeager conjugacy class $C$. There is $n_1\in \NN$ such that $k_{n_1}W$ is not meager in $V(X_n)$ thus $W^2=W^{-1}\cdot W=W^{-1}k_{n_1}^{-1}k_{n_1}W$ is not meager in $V(X_n)$ and there is $f\in C\cap W^2$. Now for any $g\in V(X_n)$, there is $h\in W^2$ such that $g$ and $h$ coincide on $X_n$.  Since $f$ is trivial outside $X_n$,
$$gfg^{-1}=hfh^{-1}\in W^6.$$
The product of two comeager subsets being everything, $V(X_n)\subset W^{12}$.

For brevity, let us denote $Y=X_n$ and $Z=\bigcup_{v\in\mathcal{V}}\bigcup_{C\in\mathcal{C}_v^T}C$. Thanks to a famous theorem of Sierpinski \cite{MR1549531}, one can find a continuum of moieties $(Y^\alpha)$ such that $Y^\alpha\subset Y$ for any $\alpha$ and $Y_v^\alpha\cap Y_v^\beta$ is finite for every $v\in\mathcal{V}$ and all $\alpha\neq\beta$. Since $|Y_v^\alpha|=|\mathcal{C}_v^T\setminus Y|$, one can find an involution $g_\alpha\in V$ such that $g_\alpha(Y_v^\alpha)=\mathcal{C}_v^T\setminus Y$ and $g_\alpha$ fixes pointwise $Y_v\setminus Y_v^\alpha$ for all $v\in\mathcal{V}$.

By the pigeonhole principle, there are $\alpha\neq\beta$ and $n_2\in\NN$ such that $g_\alpha,g_\beta\in k_{n_2}W$ and thus $g_\beta^{-1}g_\alpha \in W^2$. Let us denote $g=g_\beta^{-1}g_\alpha$ and $Y’=g Y$. One has $Y’=Z\setminus g_\beta Y^\alpha$. Thus 
$$Y\cup Y’=Z\setminus g_\beta(Y^\alpha\cap Y^\beta)\ \textrm{and}$$
$$Y\cap Y’=Z\setminus g_\beta(Y^\alpha\cup Y^\beta).$$

Thanks to the proof of the first lemma in \cite{MR859950}, $$V(Y\cup Y’)\subset V(Y)V(Y’)V(Y)\cup V(Y’)V(Y)V(Y’).$$ Since $V(Y’)=gV(Y)g^{-1}\subset W^{16}$,  $V(Y\cup Y’)\subset W^{44}$. By density of $W^2$ in $V$ and the finiteness of $Y_\alpha\cap Y_\beta$, one can find $h\in W^2\cap V$ such that $h(Z\setminus(Y\cup Y’))\subset Y\cup Y’$. If $Y\mydprime=h(Y\cup Y’)$ then $Y\cup Y’\cup Y\mydprime=Z$. So $V(Y\mydprime)=hV(Y\cup Y’)h^{-1}\subset W^{48}$ and as above $V=V((Y\cup Y’)\cup Y\mydprime)\subset W^{140}$.
\end{proof}

\begin{proof}[Proof of Lemma~\ref{E}] We rely on the proof  of the Steinhaus property for $\Aut(\QQ,<)$ \cite[Theorem 15]{Rosendal-Solecki} and use close notations. For an edge $e=\{x,y\}\in\mathcal{E}$, the group $G_e$ is isomorphic to $A$, the subgroup of $G$ fixing two end points. We also denote $D(x,y)$ by $D(e)$. 

We now define a moiety for $\cup_{e\in\mathcal{E}}D(e)$. We define  a linear order $\leq$ on $[x,y]$ that is $G_e$-invariant by $s\leq t\iff s\in[x,t]$. We choose increasing sequences $(x^e_i)_{i\in\ZZ}$ of regular points in $[x,y]$ such that $x^e_i\to y$ when $i\to+\infty$ and $x^e_i\to x$ when $i\to-\infty$. The moiety associated to this family of sequences is $$X=\bigcup_{e\in\mathcal{E},n\in\ZZ}D(x^e_{2n},x^e_{2n+1}).$$ For such a moiety $X$, we denote by $A(X)$ the subgroup of $E$ supported on $X$. As in \cite[Lemma 16]{Rosendal-Solecki}, we have 
$$E=\bigcup_{X,Y\in\mathbf{D}} A(X)A(Y)$$
where $\mathbf{D}$ is the set of moieties of $\cup_{e\in\mathcal{E}}D(e)$. 

We claim that for any $X\in\mathbf{D}$, $A(X)\subset W^{48}$, which is sufficient to prove the lemma. Let us fix some moiety $X$ and for simplicity, let us write $I_n^e=D(x^e_{2n},x^e_{2n+1})$. Thus $$X=\bigcup_{e\in\mathcal{E},n\in\ZZ}I_n^e.$$ A \emph{sub-moiety} of $X$ is a moiety of the form $$\bigcup_{e\in\mathcal{E},n\in\ZZ}I^e_{\varphi(n)}$$ where $\varphi\colon\ZZ\to\ZZ$ is injective. Using countably many disjoint sub-moieties of $X$, with a similar argument as in Lemma~\ref{V}, one get the existence of a sub-moiety $X^0$ such that $A(X^0)\subset W^{12}$. Now, choose a continuum  $(X_\alpha)$ of almost disjoint sub-moieties of $X^0$. As above, the existence of such almost disjoint sub-moieties is a consequence of \cite{MR1549531}. 

Writing $$X_\alpha=\bigcup_{e\in\mathcal{E},n\in\ZZ}I^e_{\varphi_\alpha(n)},$$ we set $J_{\alpha,2n}^e=I^e_{\varphi_\alpha(n)}=D\left(x^e_{2\varphi_\alpha(n)},x^e_{2\varphi_\alpha(n)+1}\right)$ and $J_{\alpha,2n+1}^e=D\left(x^e_{2\varphi_\alpha(n)+1},x^e_{2\varphi_\alpha(n+1)}\right)$. This way, each D(e) is the union $\cup_{n\in\ZZ}\overline{J^e_{\alpha,n}}$ and two consecutive $\overline{J^e_{\alpha,n}}$ have a unique common point that is a non-branch point. One can find $g_\alpha\in E$ such that for all $e\in\mathcal{E}$, $g_\alpha(J_{n,\alpha}^e)=J^e_{n+1,\alpha}$. There is $\alpha\ne\beta$ and $k\in G$ such $g_\alpha,g_\beta\in kW$ and thus $g_\beta^{-1}g_\alpha,g_\alpha^{-1}g_\beta\in W^2$.

If $I_n^e$ is not in the moiety $X_\alpha$ (i.e. $I_n^e\subset J^e_{\alpha,2m-1}$ for some $m\in\ZZ$) then $g_\alpha(I_n^e)\subset J^e_{\alpha,2m}$. By almost disjointness, for all but finitely many $m$, $J^e_{\alpha,2m}\subset X\setminus X_\beta$ and thus $g_\beta^{-1}(J^e_{\alpha,2m})\subset X_\beta$. So for all $n$ such that $I_n^e\notin X_\alpha$ except a finite number, $X_\alpha$ then $g
_\beta^{-1}g_\alpha(I_n^e)\subset X_\beta$. Similarly, for all $n$ such that $I_n^e\notin X_\beta$ except a finite number, $g
_\alpha^{-1}g_\beta(I_n^e)\subset X_\alpha$. Moreover, there are only finitely many $n$ such that $I_n^e\in X_\alpha\cap X_\beta$. In conclusion, for all but finitely many $n$, 
\begin{equation}\label{or}g_\beta^{-1}g_\alpha(I_n^e)\subset X_\beta\ \textrm{or} \ g
_\alpha^{-1}g_\beta(I_n^e)\subset X_\alpha.
\end{equation}
Let $n_1(e),\dots,n_k(e)$ be the indices such that the  Condition~\eqref{or} is not satisfied. By density of $W^2$ in $E$, one can find $h_e\in W^2$ such that $h(I_{n_1(e)}^e\cup\dots\cup I_{n_k(e)}^e)\subset X^0$ for all $e\in\mathcal{E}$.  Let $X^1$ be the union of all $I_n^e$ such that $g_\beta^{-1}g_\alpha(I_n^e)\subset X_\beta$, $X^2$ the union of all $I_n^e$ such that $g
_\alpha^{-1}g_\beta(I_n^e)\subset X_\alpha$ and $X^3=\cup_{e\in\mathcal{E}}I_{n_1(e)}^e\cup\dots\cup I_{n_k(e)}^e$. Since $X=X^1\cup X^2\cup X^3$, 
$$A(X)=A(X^1)A(X^2)A(X^3).$$
Moreover, each $A(X^i)$ is included in a conjugate of $A(X^0)$ by $g_\beta^{-1}g_\alpha$, $g
_\alpha^{-1}g_\beta$ or $h_e$, that are elements of $W^2$. So, $A(X^i)\subset W^{16}$ and $A(X)\subset W^{48}$.
\end{proof}

\begin{rem}A closed subgroup of $\sinf$ has the automatic continuity property as soon as the stabilizer of some point has the same property. So, to get only the automatic continuity property for $G_\infty$, it is easier to prove that the stabilizer of a branch point is Steinhaus which is a slightly simpler version of  Lemma~\ref{V}.
\end{rem}

In a Polish group $G$, an element is \emph{generic} if its conjugacy class is comeager. A group $G$ has \emph{ample generics} if for any $n\in\NN$, the diagonal conjugacy action $G\action G^n$ has a comeager orbit.  An element in $G^n$ whose orbit is comeager is also called \emph{generic}. The existence of ample generics is a very strong property and it implies the Steinhaus property \cite{Kechris-Rosendal}. Even if the group $G_\infty$ has the Steinhaus property, it does not have ample generics. Here is a more precise version of Proposition \ref{nag}.

\begin{prop}\label{nco} There is no comeager orbit in the diagonal conjugacy action $G_\infty\action G_\infty\times G_\infty$
\end{prop}

In  \cite[\S3\&6]{Kechris-Rosendal}, a framework for existence of generic elements and ample generics is introduced. The notion of \emph{turbulence} plays a key role. Let us recall the definition in the particular case of non-archimedean Polish groups. Let $G$ be a closed subgroup of $\mathcal{S}_\infty$ acting continuously on some Polish space $X$. A point $x\in X$ is \emph{turbulent}  if for any open subgroup $V\leq G$, $x\in \Int\left(\overline{V\cdot x}\right)$ that is $x$ lies in the interior of the closure of its $V$-orbit. This notion is important for us because of the following.

\begin{prop}[{\cite[Proposition~1.4]{Kechris-Rosendal}}]Let $G$ be a closed subgroup of $\sinf$ and suppose $G$ acts continuously on the Polish space $X$. Then the following are equivalent for any $x\in X$:
\begin{enumerate}
\item the orbit $Gx$ is a dense $G_\delta$-subset; 
\item the orbit $Gx$ is dense and turbulent.
\end{enumerate}
\end{prop}

\begin{proof}[Proof of Proposition \ref{nco}] We prove that the existence of a generic pair $(f,g)\in G_{\infty}^2$ under the diagonal conjugacy action would yield a generic pair in $\Aut(\QQ,<)$ and thus would contradict \cite[Theorem 2.4]{MR2354899}. 

Let $(f,g)$ be such a generic pair. Let $x,y$ be distinct branch points in $D_\infty$ and let us denote $U=\Fix(x,y)$. By density, we may assume that $(f,g)\in U^2$. Let us fix some identification $$\iota\colon\Br(D_\infty)\cap]x,y[\to(\QQ,<).$$ This yields a continuous and open surjective homorphism $$\Pi\colon U\to\Aut(\QQ,<).$$ 
Let us denote $(\varphi,\psi)$ the image of $(f,g)$ by $\Pi\times\Pi$. Any open subgroup of $\Aut(\QQ,V)$ contains some open subgroup $V=\Fix(x_1,\dots,x_n)$ with $x_1,\dots,x_n\in\QQ$. Let us set the subgroup $\widetilde{V}=\Pi^{-1}(V)$ that is $\Fix(x,\iota^{-1}(x_1),\dots,\iota^{-1}(x_n),y)$. By turbulence of $(f,g)$, we know that $(f,g)\in\Int\left(\overline{\widetilde{V}\cdot (f,g)}\right)$ and thus $(\varphi,\psi)\in\Int\left(\overline{{V}\cdot(\varphi,\psi)}\right)$. So $(\varphi,\psi)$ is turbulent. 

It remains to show that the orbit of $(\varphi,\psi)$ is dense. The orbit $G\cdot(f,g)$ is comeager and since $U$ has countable index in $G$, $U\cdot(f,g)$ is non-meager. Moreover, this orbit is included in $U^2$, so it is non-meager in $U^2$. Thus, there is a (non-empty) basic open set $\widetilde{V}_{\mathbf{x,y,z}}$ of $U^2$ such that $U\cdot(f,g)$ is dense in $\widetilde{V}_{\mathbf{x,y,z}}$ where $$
\widetilde{V}_{\mathbf{x,y,z}}=\left\{(f’,g’)\in U^2; \ f’(x_i)=y_i,\ g’(x_i)=z_i,\ \forall i\in\{1,\dots,n\}\right\}$$
and
$$\mathbf{x}=(x_1,\dots,x_n)\in \left(\Br(D_\infty)\setminus\{x,y\}\right)^n,$$
$$\mathbf{y}=(y_1,\dots,y_n)\in \left(\Br(D_\infty)\setminus\{x,y\}\right)^n,$$
$$\mathbf{z}=(z_1,\dots,z_n)\in \left(\Br(D_\infty)\setminus\{x,y\}\right)^n.$$
Now a basis of  open subsets of $\Aut(\QQ,<)^2$ is given by subsets 
$$V_{\mathbf{p,q,r}}=\left\{(\varphi’,\psi’);\ \varphi’(p_i)=q_i,\ \psi’(p_i)=r_i,\ \forall i\in\{1,\dots,m\}\right\}$$
where $\mathbf{p,q,r}$ are $m$-tuples of distinct points in $\QQ$. Let $z\in]x,y[$ such that $]z,y[$ does not contain any image of elements of $\mathbf{x,y,z}$ by the retraction $D_\infty\to[x,y]$. Up to conjugate by an element of $U$, we may assume that $\{\iota^{-1}(t_i),\iota^{-1}(p_i),\iota^{-1}(r_i),; \ i\in\{1,\dots,m\}\}$ is included in $]z,y[$ and thus $\widetilde{V}_{\mathbf{x,y,z}}\cap\widetilde{V}_{\iota^{-1}(\mathbf{p}),\iota^{-1}(\mathbf{q}),\iota^{-1}(\mathbf{r})}$ is a non-empty open subset that meet $U\cdot(f,g)$. This implies that the orbit of $(\varphi,\psi)$ meets $V_{\mathbf{p,q,r}}$.
\end{proof}
\section{Universal minimality of the topology}\label{umin}

Let us recall that on any set $X$, the set of topologies is partially ordered by finess. For two topologies $\tau_1,\tau_2$, $\tau_1<\tau_2$ ($\tau_2$ is finer than $\tau_1$) if and only if $\tau_1\subseteq\tau_2$ (as subsets of $2^X$).

For a Hausdorff topological group $(G,\tau)$, one say that the group is \emph{minimal} if $\tau$ is minimal among Hausdorff group topologies on $G$ and the topological group is \emph{universally minimal} if it is a least element. The goal of this section is to show that $G_S$ with its \emph{natural topology} is universally minimal. By the natural topology, we mean the compact-open topology associated to the action on the dendrite and let us recall that it coincides with the non-archimedean one coming from the action on the set of branch points. For this natural topology, the stabilizers of branch points are open subgroups and generate the topology.\\

Let us fix $S\subseteq\overline{N}_{\geq3}$ and let us denote $\mathcal{U}$ the collection of all $D(x,y)$, the unique connected component of $D_S\setminus\{x,y\}$ that contains $\left]x,y\right[$, where $x$ and $y$ are distinct branch points. For $U,V\in\mathcal{U}$, let $$O(U,V)=\{g\in G_S,\ g(U)\cap V\neq\emptyset\}.$$

Let us fix some Hausdorff group topology $\tau$ on $G_S$. For the remainder of this section, any topological property on $G_S$ is with respect to $\tau$. The starting point is the standard fact that centralizers $C_G(g)$ of any element $g$ are closed in any Hausdorff topological group $G$.

\begin{lem} For any $U,V\in\mathcal{U}$, $O(U,V)$ is open.
\end{lem}

\begin{proof} The complement of $O(U,V)$ is $C(U,V^\mathsf{c})=\{g\in G_S,\ g(U)\subseteq V^\mathsf{c}\}$. Following an observation due to Kallman \cite[Theorem 1.1]{MR831205}, $C(U,V^\mathsf{c})$ is closed. Actually, we claim that 
$$C(U,V^\mathsf{c})=\bigcap_{g,h}\{f\in G_S,\ fgf^{-1}\in C_{G_S}(h)\}$$
where $g$ ranges over all element with support in $\overline{U}$ (equivalently in $U$) and $h$ ranges over all elements with support in $V$.

Let $k\in G_S$. Assume there is $x\in V$ such that $k(x)\neq x$, then one can find $h\in G_S$ with support on $V$, fixing $f(x)$ and not $x$. Thus an element $k$ commutes with all elements $h$ supported on $V$ if and only if $\supp(k)\subset V^\mathsf{c}$. Since $\supp(fgf^{-1})=f(\supp(g))$, $f\in C(U,V^\mathsf{c})$ if and only if for all $g$ supported on $U$, $f(\supp(g))\subset V^\mathsf{c}$, that is $f(U)\subset V^\mathsf{c}$.
\end{proof}

\begin{proof}[Proof of Theorem \ref{tumin}] It suffices to show that for any $x\in\Br(D_S)$, $\Fix(x)$ is open. 

Let us fix some branch point $x$ and let $\mathcal{U}_{x,3}$ be the subset of $\{U=(U_1,U_2,U_3)\in\mathcal{U}^3\}$ such that the $U_i$'s lie in distinct components of $D_S\setminus\{x\}$. Observe that if $U\in\mathcal{U}_{x,3}$ and $g\in\Fix(x)$ then $g(U_1),g(U_2), g(U_3)$ lie in 3 distincts components of $D_S\setminus\{x\}$. Moreover, the following converse holds: if for some  $U\in\mathcal{U}_{x,3}$ and each $i$,  $g(U_i)$ intersects some $V_i\in\mathcal{U}$ where $V\in\mathcal{U}_{x,3}$ then $g\in\Fix(x)$. To see this fact, choose $x_i\in f^{-1}(V_i)\cap U_i$ for each $i$ then $x$, the center of $[x_1,x_2,x_3]$ is also the center of $[f(x_1),f(x_2),f(x_3)]$ and thus $f(x)=x$.
So  Now, $\Fix(x)$ is open because 
$$\Fix(x)=\bigcup_{U,V\in\mathcal{U}_{x,3}}O(U_1,V_1)\cap O(U_2,V_2)\cap O(U_3,V_3).$$
\end{proof}

\section{Small index subgroups}
Let us recall that a Polish group has the \emph{small index property} if any subgroup of small index, i.e. of index less than $2^{\aleph_0}$, is open. For example $\mathcal{S}_\infty$ and $\Aut(\QQ,<)$ have this property.

Let us start with an example that will be useful for us. Let us denote by $G_\xi$ the stabilizer in $G_\infty$ of some end point $\xi\in D_\infty$. 

\begin{prop}\label{sip} The Polish group $G_\xi$ has the small index subgroup property.
\end{prop}

\begin{proof} This is a consequence of  \cite[Theorem 4.1]{MR988099}. This theorem states that the automorphism group of a countable 2-homogeneous tree which is a meet-semilattice has the small index property. 

Let us consider the countable set $\Br(D_\infty)$ endowed with the order $x\leq_\xi y\iff x\in[\xi,y]$. As it appears in \cite[Example 5.2]{DM_dendritesII}, $(\Br(D_\infty),\leq_\xi)$ is dense semi-linear order and it is a meet semi-lattice where the meet  of $a,b\in\Br(D_\infty)$, that is the infimum of $\{a,b\}$, is  $a\wedge b=c(a,b,\xi)\in\Br(D_\infty)$. Moreover it is 2-homogeneous, that is any isomorphism between two subsets with 2 elements extends to an isomorphism of $(\Br(D_\infty,\leq_\xi)$. Actually, for $a,b,a',b'\in \Br(D_\infty)$, if $(\{a,b\},\leq_\xi)$ and  $(\{a',b'\},\leq_\xi)$ are isomorphic then the labeled graphs $\langle \{a,b,\xi\}\rangle$ and $\langle \{a',b',\xi\}\rangle$ are isomorphic and one can find $g\in G_\infty$ that induces this partial isomorphism by Proposition \ref{prop:extension}. 

Now, by \cite[Corollary 5.21]{DM_dendritesII}, $\Aut(\Br(D_\infty),\leq_\xi)\simeq G_\xi$ and thus $G_\xi$ has the small index property. 
\end{proof}

Let $\Omega$ be a countable infinite set with full permutation group $\mathcal{S}_\infty$. For a group $G$, we denote by $G\wr \sinf$ the (unrestricted permutational) wreath product $G^\Omega\rtimes\sinf$. The action of $\sinf$ on $G^\Omega$ is by permutation of the coordinates. If $\sigma\in\mathcal{S}_\infty$ and $(g_\omega)_\omega\in G^\Omega$ then $$\sigma\cdot(g_\omega)=(g_{\sigma^{-1}\omega})_\omega.$$ If the role of $\Omega$ shall be emphasized, we denote the above wreath product $G\wr_\Omega \sinf$.

%
%

In the particular case where $G$ is a closed subgroup of $\sinf$ acting on a countable set $\Lambda$ and another copy of $\sinf$ acts as above on the countable set $\Omega$, the wreath product $G\wr\sinf$ acts on $\Lambda\times \Omega$ with the \emph{imprimitive action}. This action is given by the following formula:
$$\left((g_\omega),\sigma\right)\cdot(\lambda,\omega’)=\left(g_{\sigma(\omega’)}\lambda,\sigma(\omega’)\right).$$
This action embeds $G\wr\sinf$ as a closed subgroup of the symmetric group of $\Lambda\times \Omega$ and thus it has a natural Polish topology and we will  consider this group with this topology. 
\begin{thm}\label{wr} Let $G$ be a closed subgroup of $\sinf$ with a comeager conjugacy class and the small index property. The wreath product $W=G\wr\sinf$ is a Polish group with the small index property.\end{thm}
\begin{lem}\label{sisf}Let $G$ be some  closed subgroup of $\sinf$. The group $G$ has the small index property if and only if the stabilizer of any point in $G$ has the small index property.
\end{lem}
Before proving Theorem~\ref{wr} and Lemma~\ref{sisf}, let us see how they imply Theorem~\ref{sis}, that is $G_\infty$ has the small index property.
\begin{proof}[Proof of Theorem~\ref{sis}] Thanks to Lemma~\ref{sisf}, we know that $G_\infty$ has the small index  property if and only if the stabilizer $G_b$ of some branch point $b$ has the same property. Since $G_b$ is isomorphic $G_\xi\wr_{\mathcal{C}_b}\sinf$ where $G_\xi$ (see \cite[Lemma 7.1]{DM_dendritesII}), the theorem is a consequence of Theorem~\ref{wr} and Proposition~\ref{sip}.
\end{proof}

\begin{proof}[Proof of Lemma~\ref{sisf}] Let $H$ be a subgroup of small index of $G$ and $G_x$ be the stabilizer of some point $x\in\Omega$. One has $|G_x\colon H\cap G_x|\leq|G\colon H|<2^{\aleph_0}$. So if $G_x$ has the small index  property then $H\cap G_x$ is open in $G_x$ so $H\cap G_x$ is open in $G$ and thus $H$ is open in $G$.

Conversely, let $H$ be a subgroup of small index of $G_x$. By the index formula, 
$$|G\colon H|= |G\colon G_x|\, |G_x\colon H|.$$
Since $\Omega$ is countable $|G\colon G_x|\leq\aleph_0$ and thus $|G\colon H|<2^{\aleph_0}$. So $H$ is open in $G$ and thus in $G_x$.
\end{proof}

\begin{lem}\label{nsi} Let $G$ be a Polish group with a comeager conjugacy class. If $N$ is a normal subgroup of small index then $N=G$.
\end{lem}

\begin{proof}Let $\mathcal{C}$ be the comeager conjugacy class. It suffices to show that $\mathcal{C}\cap N\neq\emptyset$. Otherwise by normality, $\mathcal{C}\subset N$ and thus $G=\mathcal{C}\cdot\mathcal{C}\subset N$. 

Since $N$ has small index then $N$ is not meager \cite[Theorem 4.1]{MR1231710} (see also \cite[Lemma 6.8]{Kechris-Rosendal} for a very short proof), so $N\cap\mathcal{C}\neq\emptyset$ and we are done.\end{proof}

Our proof that $G\wr\sinf$ has the small index subgroup property borrow the original ideas that lead to prove the property for $\sinf$ \cite[Theorem 1]{MR859950}. We not only use the result but also the proof itself and thus reproduce some of the arguments there. Let us recall that a moiety of $\Omega$  is subset $\Sigma$ that is infinite and co-infinite.

\begin{proof}[Proof of Theorem~\ref{wr}] Let $H$ be a subgroup of $W$ of small index. Let $(\Sigma_i)$ be an infinite collection of disjoint moieties of $\Omega$. Let $W_i=G\wr_{\Sigma_i} \Sym(\Sigma_i)\leq\Sym(\Lambda\times \Omega)$ be the subgroup of permutations supported on $\Lambda\times\Sigma_i$. More precisely, $\left((g_\omega),\sigma\right)\in W_i$ if $\supp(\sigma)\subset \Sigma_i$ and $g_\omega=e$ for $\omega\notin \Sigma_i$. By disjointness of the supports, for $i\neq j$, $W_i\cap W_j$ is trivial and these two subgroups commute. Let $P$ be the product subgroup $\prod_i W_i\leq\Sym(\Lambda\times \Omega)$. Let $H_i$ be the projection of $H\cap P$ on $W_i$. We have

$$\prod_{i}|W_i\colon H_i|=|P\colon \prod_{i}H_i|\leq|P\colon H\cap P|\leq|W\colon H|<2^{\aleph_0}.$$

This implies that for all but a finite number, $W_i=H_i$. We fix such an $i$ and simply note $\Sigma=\Sigma_i$ and $W’=W_i$. We denote by $G^\Sigma$ the subgroup of $G\wr\sinf$ with elements of the form $((g_\omega),e)$ and $g_\omega=e$ for $\omega\notin \Sigma$. Let $g\in H\cap G^\Sigma$ and $g’\in G^\Sigma$, there is $h\in H\cap P$ such that $\pi_i(h)=g’$ where $\pi_i\colon P\to W_i$ is the projection. One has $g’g(g’)^{-1}= hgh^{-1}\in H\cap G^\Sigma$. So, the subgroup $H\cap G^\Sigma$ is a normal subgroup of $G^\Sigma$ of small index. The group $G^\Sigma$ has a comeager conjugacy because of  \cite[Lemma 11]{Rosendal-Solecki} and the fact that $G$ has a comeager conjugacy class. By Lemma~\ref{nsi}, $H\cap G^\Sigma=G^\Sigma$. 

We denote by $\Sym(\Sigma)$ the subgroup of $G\wr\sinf$ of elements of the form $((g_\omega),\sigma)$ where $g_\omega=e$ for all $\omega\in\Omega$ and $\supp(\sigma)\subset \Sigma$. Since $\Sym(\Sigma)\cong\sinf$ and the non-trivial normal subgroups of $\sinf$ are the finitary symmetric and alternating subgroups, which are of index $2^{\aleph_0}$, we know that for $H\cap \Sym(\Sigma)=\Sym(\Sigma)$. Now, $W’=G\wr_\Sigma\Sym(\Sigma)$ is generated by $G^\Sigma$ and $\Sym(\Sigma)$. Thus $H\geq W’$. 

Following the proof of \cite[Theorem 1]{MR859950}, we choose a continuum of almost disjoint moieties $(\Sigma_\alpha)$ of $\Sigma$ and an involution $g_\alpha\in\sinf$ exchanging $\Sigma_\alpha$ with $\Omega\setminus\Sigma$ and fixing pointwise $\Sigma\setminus\Sigma_\alpha$. By the pigeonhole principle, for some $\alpha\neq\beta$, $g_\alpha$ and $g_\beta$ are in the same $H$-class and thus $g=g_\beta^{-1} g_\alpha$ is in $H$. Let $\Sigma’=g(\Sigma)$, then $G\wr_{\Sigma’}\Sym(\Sigma’)=gW’g^{-1}\leq H$. Since $\Sigma\cup\Sigma’=\Omega\setminus g_\beta(\Sigma_\alpha\cap\Sigma_\beta)$ and $\Sigma\cap\Sigma’=\Omega\setminus g_\beta(\Sigma_\alpha\cup\Sigma_\beta)$ are infinite, the first lemma of  \cite{MR859950} shows that $H_0=G\wr_{\Sigma\cup\Sigma’}\Sym(\Sigma\cup\Sigma’)\subset H$. 

Let us denote by $F$ the finite set $g_\beta(\Sigma_\alpha\cap\Sigma_\beta)$. For $f\in F$, let $G_f\leq G\wr_\Omega\sinf$ be the corresponding copy of $G$ acting on $\Lambda\times\{f\}$. The subgroup $H_f=H\cap G_f$ has small index and thus is open in $G_f$. Now 

$$H\geq\left(\Pi_{f\in F}H_f\right)\times H_0.$$
The right hand side being open (containing the pointwise stabilizer of a finite number of points), $H$ is open as well.
%
%
\end{proof}

%

\section{Universal minimal flow}\label{umf}

The goal of this section is to identify, in Theorem~\ref{univminflow}, the universal minimal flow of $G_\infty$. For a general topological group $G$ (for example a locally compact group), this universal minimal flow $M(G)$ is a huge compact space, often non-metrizable. After \cite{Pestov98,MR1937832,MR2140630} a general framework emerged to identify universal minimal flows of some Polish groups. One may have a look to \cite{MR2277969} for a survey.

 Let us recall that a Hausdorff topological group is \emph{extremely amenable} if any continuous action on a compact space has a fixed point. The idea to identify the universal minimal is to find an extremely amenable subgroup $G^*$ of the Polish group $G$ and consider the completion $\widehat{G/G^*}$ (for the quotient of the right  uniform structure on $G$) of $G/G^*$. The extreme amenability of $G^*$ can be obtained thanks to a Ramsey property.  We recommend \cite{melleray2015polish,MR3613453} for more details about this strategy and for relations between the existence of a comeager orbit in $M(G)$ and the metrizability of $M(G)$.
 
 We follow this strategy and we introduce the tools we use below.

\begin{defn} Let $X$ be a dendrite. A linear order $\prec$ is on $\Br(X)$ is \emph{converging} if for any $x,y,z\in\Br(X)$, $y\in]x,z[\implies y\prec x\ \mathrm{or}\ y\prec z$.
\end{defn}

The meaning of the adjective converging appears in the following lemma: minimizing sequences converge to a unique point that we call the root below.
\begin{lem}\label{min}Let $X$ be a dendrite and $\prec$ be a converging linear ordering on $\Br(X)$. There is a unique point $x_0\in X$ which is the limit of any minimizing sequence. Moreover, for any $a,b\in \Br(X)$, if $a\in[x_0,b]$ then $a\preceq b$.
\end{lem}

\begin{proof} Let $(x_n)$ be a minimizing sequence (i.e. for any $x\in\Br(X)$, there is $N$ such that for any $n\geq N$, $x_n\preceq x$). By compactness, this sequence has at least one adherent point in $X$. Assume there are two adherent points, that is we have two subsequences $x_{\varphi(n)},x_{\psi(n)}$ converging respectively to $x_\varphi$ and $x_\psi$. Let $x\in]x_\varphi,x_\psi[$. For $n$ large enough, $x\in]x_{\varphi(n)},x_{\psi(n)}[$ and thus $x\prec x_{\varphi(n)}$ or $x\prec x_{\psi(n)}$. So we have a contradiction and $(x_n)$ converges to some point $x_0$. Replacing $(x_{\varphi(n)})$ and $(x_{\psi(n)})$ by any two minimizing sequences, the same argument shows that  the limit point $x_0$ is independent of the choice of the minimizing sequence. 

Now, let $a,b\in \Br(X)$ with $a\in[x_0,b]$. For $n$ large enough $a\succeq x_n$ and $a\in[x_n,b]$ thus $a\preceq b$.
\end{proof}

The point $x_0$ is called the \emph{root} of the converging linear order $\prec$.
\begin{defn} Let $X$ be a dendrite and $\prec$ a converging linear order with root $x_0$. The order $\prec$ is \emph{convex} if for $a,b\in\Br(X)$, $c=c(a,b,x_0)$,  $a'\in[a,c]$  and $b'\in[b,c]$,$$ a\prec b\implies a'\prec b’.$$ \end{defn}

\begin{rem}Let us observe that a general converging linear order is not necessarily convex. Let us consider the  simple dendrite $D$ in Figure~\ref{notconvex} with the converging linear order $\prec$ such that
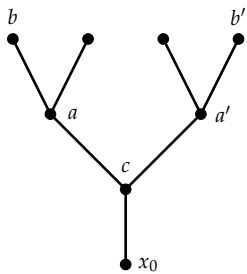
\begin{figure}
\begin{tikzpicture}[line cap=round,line join=round,>=triangle 45,x=1.0cm,y=1.0cm]
\clip(0.3255133394532428,-0.1927511237105577) rectangle (3.8177891182026134,3.438421987148181);
\draw [line width=1.pt] (2.,0.)-- (2.,1.);
\draw [line width=1.pt] (1.,2.)-- (1.5,3.);
\draw [line width=1.pt] (1.,2.)-- (0.5,3.);
\draw [line width=1.pt] (2.,1.)-- (1.,2.);
\draw [line width=1.pt] (2.,1.)-- (3.,2.);
\draw [line width=1.pt] (3.,2.)-- (2.5,3.);
\draw [line width=1.pt] (3.,2.)-- (3.5,3.);
\begin{scriptsize}
\draw [fill=black] (2.,0.) circle (2pt);
\draw[color=black] (2.3,0) node {$x_0$};
\draw [fill=black] (2.,1.) circle (2pt);
\draw[color=black] (2,1.3) node {$c$};
\draw [fill=black] (1.,2.) circle (2pt);
\draw[color=black] (1.3,2) node {$a$};
\draw [fill=black] (1.5,3.) circle (2pt);
\draw [fill=black] (0.5,3.) circle (2pt);
\draw[color=black] (0.5,3.3) node {$b$};
\draw [fill=black] (3.,2.) circle (2pt);
\draw[color=black] (3.3,2) node {$a’$};
\draw [fill=black] (2.5,3.) circle (2pt);
\draw [fill=black] (3.5,3.) circle (2pt);
\draw[color=black] (3.5,3.3) node {$b’$};
\end{scriptsize}
\end{tikzpicture}
\caption{The simple dendrite $D$ with a converging but not convex linear order $\prec$.}\label{notconvex}
\end{figure}
\begin{itemize}
\item $x_0$ is the root,
\item $c\prec a,a’,b,b’$,
\item $a,a’\prec b,b’$,
\item $a\prec a’$,
\item $b’\prec b$.
\end{itemize}

The conditions $a\prec a’$ and $b’\prec b$ show that this order is not convex.

\end{rem}
We denote by $\CLO(X)$ the set of convex converging linear orders on $\Br(X)$. It is clear from the definition that $\CLO(X)$ is a metrizable $\Homeo(X)$-flow since it is a closed invariant subspace of the space of all linear orders $\LO(\Br(X))$ on $\Br(X)$ which is compact for the pointwise convergence.

We observe that a convex converging linear order $\prec$ induces an linear order $\prec^x$ on the connected components around a given point $x$.

\begin{lem}\label{obsorder}Let $x\in X$, $\prec\in CLO(X)$ with root $x_0$. Let $C,C’\in\mathcal{C}_x$ distinct and that do not contain the root. Then,

$$(\forall c\in C,\ \forall c’\in C',\ c\prec c')\lor (\forall c\in C,\ \forall c’\in C',\ c\succ c’).$$ \end{lem}

\begin{proof} Let $y_0$ be the image of the root $x_0$ by the first point map to the subdendrite $C\cup\{x\}\cup C'$. Since $x_0\notin C\cup C’$, $y_0=x$.

Choose $a\in C$ and $b\in C'$. Assume that  $a\prec b$. Now by convexity of the order, for any $c\in C$ and $c'\in C'$, Let $a'=c(c,a,x)$ and $b'=c(b,c',x)$. By convexity, $a'\prec b'$ and thus $c\prec c'$.
\end{proof}

In the first case, we write $C\prec^x C’$ and otherwise we write $C’\prec^x C$. This defines a linear order on $\mathcal{C}_x$ if $x=x_0$ and on $\mathcal{C}_x\setminus C_x(x_0)$ if $x\neq x_0$.

\begin{rem}Let us observe that convex and converging linear orders have the following stability property: If $\prec\in\CLO(X)$ and $Y$ is a subdendrite of $X$ then $\prec |_{\Br(Y)}\in\CLO(Y)$. If $x_0$ is the root of $\prec$ then $\pi_Y(x_0)$ is the root of $\prec |_{\Br(Y)}$.
\end{rem}
We will see  in Theorem~\ref{univminflow} that the universal minimal $G_\infty$-flow is $\CLO(D_\infty)$. For the remaining of this section, we fix some $\xi\in\Ends(D_\infty)$. For a branch point $c$, let us denote by $\mathcal{C}_{c,\xi}$ the space $\mathcal{C}_c\setminus C_c(\xi)$.
\begin{lem}\label{cclo} For each branch point $c$, fix a linear order $\prec^c$ on the set $\mathcal{C}_{c,\xi}$ that is isomorphic to $\QQ$ with its standard linear order <. Then there is a convex converging linear order $\prec _0$ on $D_\infty$ such that the root is $\xi$ and for any branch point $c$, the linear order induced on the components of $\mathcal{C}_{c,\xi}$  is $\prec^c$.
\end{lem} 
\begin{proof} We define $\prec_0$ in the following way: for $a\neq b\in\Br(D_\infty)$, if $c(a,b,\xi)=a$ then $a\prec_0 b$ (and $b\prec_0 a$ if $c(a,b,\xi)=b$). If $c=c(a,b,\xi)\neq a,b$ then $a\prec_0b\iff C_c(a)\prec^c C_c(b)$.

Let us check it is a convex converging linear order. Totality and antisymmetry are immediate. Let $a,b,c\in\Br(D_\infty)$ such that $a\prec_0 b\prec_0 c$ and let us note $d=c(a,b,\xi)$ and $e=c(b,c,\xi)$. There are three (mutually exclusive) possibilities: $e\in]\xi,d[$, $d=e$ or $e\in]d,b]$. In the first case $a\in C_e(b)$, $C_e(a)\prec^e C_e(c)$. In the second one, $a\in[\xi,c[$ (if $a=d$) or $C_d(a)\prec^d C_d(e)$ and in the last one $C_d(a)\prec^d C_d(c)$. So, in all cases, $a\prec_0 c$.

This order is converging because if $a,b\in\Br(D_\infty)$ and $c=c(a,b,\xi)$ then $a$ is maximum on $[c,a]$ and $b$ is a maximum on $[c,b]$. Finally this order is convex by construction.
\end{proof}

\begin{rem}The order $\prec_0$ depends a priori on the choice of the linear orders $\prec^c$ for all branch points. Actually, a different choice of orders isomorphic to $(\QQ,<)$ leads to an order $\prec$ such that there is $g\in G_\infty$ fixing $\xi$ with $x\prec y\iff g(x)\prec_0 g(y)$ for all $x,y\in\Br(D_\infty)$. This can be obtained thanks to a back and forth argument on $\Br(D_\infty)$. In what follows, we will not use that fact and we will fix the order $(\prec^c)_{c\in\Br(D_\infty)}$ in the proof of Proposition~\ref{fraise} and thus  we will forgot the dependency on this choice.\end{rem}

\begin{prop}\label{Ramsey} The group $\Stab_{G_\infty}(\prec _0)$ is extremely amenable.
\end{prop}

To prove the extreme amenability of this group, we used the seminal idea that a closed subgroup of $S_\infty$ is extremely amenable if and only if it is the automorphism group of some Fraïssé limit of a Fraïssé order class with the Ramsey property \cite[Theorem 4.7]{MR2140630}. We now describe the Fraïssé class and the Ramsey theorem needed to prove Proposition~\ref{Ramsey}. We  essentially follow \cite{MR3436366}.

A (meet) \emph{semi-lattice} is a poset $(A,\leq)$ such that for any two elements $a,b\in A$, the pair $\{a,b\}$ has a greatest lower bound (that is an infimum) denoted by $a\wedge b$ and called the \emph{meet} of $a$ and $b$. It satisfies the following three properties for all $a,b,c\in A$:
\begin{itemize}
\item $a\wedge a=a$,
\item $a\wedge b=b\wedge a$ and 
\item $(a\wedge b)\wedge c=a\wedge(b\wedge c)$.
\end{itemize}
Actually, from a binary operation $\wedge$ satisfying the above three properties, one can recover the partial order $\leq$ by defining $a\leq b\iff a\wedge b=a$.
A semi-lattice $(A,\leq,\wedge)$ is \emph{treeable} if it has a minimum called the \emph{root} and all the sets $\da a=\{b\in A;\ b\leq a\}$ are linearly ordered. 

A linear order $\prec$ on a treeable semi-lattice $(A,\leq,\wedge)$ is a \emph{linear extension} of $\leq$ if $a< b\implies a\prec b$ and it is \emph{convex} if for any $a,a’,b,b’\in A$ such that $a\wedge b\preceq a’\preceq a$ and $a\wedge b\preceq b’\prec b$, $a\prec b\iff a’\prec b’$.

We denote by $\mathcal{CT}$ the class of finite treeable semi-lattices with a convex linear extension $(A,\leq,\wedge,\prec)$. 

\begin{rem} If $(A,\leq_A,\wedge_A,\prec_A)$ and $(B,\leq_B,\wedge_B,\prec_B)$ are elements of $\mathcal{CT}$, an embedding of $A$ in $B$ is an injective map $\varphi\colon A\to B$ such that for all $a,a’\in A$, $\varphi(a\wedge_A a’)=\varphi(a)\wedge_B\varphi(a’)$ and $a\prec_A a’\implies \varphi(a)\prec_B\varphi(a’)$.

We emphasize that this notion of embeddings does not coincide with the notion of embeddings for graphs. In our situation, once can add a vertex in the middle of an edge and this is impossible for graphs embeddings.
\end{rem}

Let us introduce the following partial order on $\Br(D_\infty)$: $a\leq b\iff a\in[\xi,b]$.

\begin{lem}\label{lem:tr} The poset $(\Br(D_\infty),\leq)$ is a treeable semi-lattice and $\prec_0$ is a convex linear extension of $\leq$.
\end{lem}

\begin{proof} It is straightforward to check that it is a treeable semi-lattice with meet $a\wedge b=c(a,b,\xi)$ for any $a,b\in\Br(D_\infty)$. The fact that $\prec_0$ is a convex linear extension of $<$ follows from the properties given in Lemma~\ref{cclo}.
\end{proof}

\begin{prop}\label{fraise}The Fraïssé limit of $\mathcal{CT}$ is $(\Br(D_\infty),\leq,\wedge,\prec_0)$.\end{prop}

To prove this proposition, we rely on the relation between semi-linear orders and dendrite with a chosen end point developed in \cite[\S 5]{DM_dendritesII}. A partially ordered set $(X,\leq)$ is a \emph{semi-linear} order if for any $x,y\in X$, there $z\in X$ such that $z\leq x,y$ and for all 
$x\in X$, the downward chain $\da x=\{y\in X,\ y\leq x\}$ is totally ordered. Treeable semi-lattices are particular cases of  semi-linear orders. 

A partially ordered set  $(X,\leq)$ is \emph{dense} if for all $x,y$ such that $x<y$ there is $z\in X$ and $x<z<y$. An important point of \cite[\S 5]{DM_dendritesII} is to show that a countable dense semi-linear order $T$ can be canonically embedded in some dendrite $\widehat{T}$ and the order is given by some end point as in Lemma~\ref{lem:tr}.

For example the order on $D_\infty$ defined by $x\leq y\iff [\xi,x]\subseteq[\xi,y]$, is a dense semi-linear order. 

\begin{proof}[Proof of Proposition~\ref{fraise}] It is know that $\mathcal{CT}$ is a Fraïssé class \cite[\S4]{MR3436366}. Let $(\mathbf{CT},\leq,\wedge,\prec)$ be its Fraïssé limit. One checks easily that $(\mathbf{CT},\leq)$ is a dense semi-linear order.  The density actually follows from the amalgamation property of the Fraïssé limit: for any $x,y$, such $x<y$, one can find $z$ such $x<z<y$. By \cite[Proposition 5.15 \& Theorem 5.19]{DM_dendritesII}, $\mathbf{CT}$ can be embedded in a semi-linear order $D=\reallywidehat{\mathbf{CT}}$ that can be topologized to be a dendrite. To show that $D\simeq D_\infty$, we use the characterization  of $D_\infty$. This is the only dendrite without free arc such that all branch points have infinite order. One can conclude by showing that $\mathbf{CT}$ is exactly the set of branch points of $D$, this set is arcwise dense and that any branch point has infinite order. Let us show these properties.

Recall that $D$ is the set of \emph{full down-chains} of $(\mathbf{CT},\leq)$ where a chain is a totally ordered subset of $\mathbf{CT}$. A chain $C$ is a \emph{down chain} if $x\in C\implies \da x\subset C$ and it is \emph{full} if  it contains its supremum or has no supremum in $\mathbf{CT}$. The set $\mathbf{CT}$ is embedded in $D$ via the map $x\mapsto\da x$. 

We now identify arcs and branch points in $D$. Let $C_1,C_2\in D$ and $C=C_1\cap C_2$ which is the infimum $C_1\wedge C_2$. It appears in the proof of \cite[Theorem 5.19]{DM_dendritesII} that for any $C\in D$, $\{x\in D,\ x\leq C\}$ is exactly the arc from the point $C$ of $D$ to the minimum. Since $C\subset C_1$, $\{x\in D, C\leq x\leq C_1\}$ is the arc $[C_1,C]$. Similarly $[C_2,C]= \{x\in D, C\leq x\leq C_2\}$. Since these arcs have an intersection reduced to $C$, $[C_1,C_2]=[C_1,C]\cup[C,C_2]$. So, we deduce that for any $C_1,C_2,C_3\in 
D$, the center $c(C_1,C_2,C_3)$ is $C_i\wedge C_j$ for some $i,j\leq 3$ (it is actually the maximum among $C_1\wedge C_2,C_2\wedge C_3$ and $C_3\wedge C_1$).

Let $C$ be some branch point in $D$. There are  $C_1,C_2,C_3$  such that $C=c(C_1,C_2,C_3)$ and $C\neq C_i$ for all $1\leq i\leq 3$.  There are $i,j$ such that $C=C_i\wedge C_j$. Choose $x\in\mathbf{CT}$ such that $C<x<C_i$ and $C<y<C_j$. So $C=C_i\wedge C_j=x\wedge y\in\mathbf{CT}$. To conclude that $D\simeq D_\infty$, it remains to show that a point of $\mathbf{CT}$ has infinite order. Let us consider the following finite treeable semilattice $(A,\leq,\prec)$: $A=\{x_0,x_1,\dots,x_n\}$, $x_0\leq x_i$ for all $i\leq n$, $x_i,x_j$ are incomparable for $\leq$ if $i,j\neq0$ and $\prec$ is any linear order such that $x_0$ is a minimum.  For any embedding $\varphi$ of $A$ in $D$ via $\mathbf{CT}$,  the image of $x_0$ has order at least $n$ because the $x_i$'s are mapped to different connected components of $D\setminus\{\varphi(x_0)\}$ because $\varphi(x_i)\wedge\varphi(x_j)=\varphi(x_0)$ and thus $\varphi(x_0)\in[\varphi(x_i),\varphi(x_j)]$. Since $\mathbf{CT}$ is homogeneous, it means that each point of $\mathbf{CT}$ has order at least $n$ and thus all these points have infinite orders. Let us finish the proof by showing that $\mathbf{CT}$ is arcwise dense. Let $C_1\neq C_2\in D$. We know that $[C_1,C_2]=[C_1,C]\cup[C,C_2]$ where $C=C_1\wedge C_2$. Since $C_1\neq C_2$, there is $i$ such that $C_i\neq C$. So, $C<C_i$, and since $C$ and $C_i$ are full down-chains, there is $x\in\mathbf{CT}$ that belongs to $C_1$ and not to $C$. So $x\in[C_1,C_2]$ and $D$ is arcwise dense. 

So $(\widehat{\mathbf{CT}},\leq)\simeq (D_\infty,\leq)$ and $\mathbf{CT}$ corresponds to $\Br(D_\infty)$ in this identification. Let $c\in\Br(D_\infty)$. Two points $a,b\in\Br(D_\infty)$ such that $a,b>c$, are in the same connected of $\mathcal{C}_c$ if and only if $c(\xi,a,b)\neq c$ that is if and only if $a\wedge b>c$. Since $\prec$ is convex, it induces a dense linear oder $\prec^c$ on elements of $\mathcal{C}_c$ that do not contain $\xi$ . By the amalgamation property each order $\prec^c$ is countable and dense thus isomorphic to $(\mathbf{Q},<)$ and $\prec_0$ is obtained by Lemma~\ref{cclo}.
\end{proof}

%
%
%

\begin{lem}\label{groupiso} The groups $\Aut(\Br(D_\infty),\leq,\wedge,\prec_0)$ and $\Stab_{G_\infty}(\prec_0)$ are isomorphic.\end{lem}

\begin{proof} It is proved in \cite[Corollary 5.21]{DM_dendritesII} that $\Aut(\Br(D_\infty),\leq)\simeq \Stab_{G_\infty}(\prec_0)$ and the lemma follows.
\end{proof}

We can now prove Proposition~\ref{Ramsey}.
\begin{proof} [Proof of Proposition~\ref{Ramsey}] Thanks to \cite[Theorem 4.7]{MR2140630}, it suffices to show that the group $\Stab_{G_\infty}(\prec_0)$ is the automorphism group of some Fraïssé limit of some Fraïssé order class with the Ramsey property. By Lemma~\ref{groupiso}, $\Stab_{G_\infty}(\prec_0)$ is the automorphism group  of the limit of the class $\mathcal{CT}$ and this class has the Ramsey property \cite[Theorem2]{MR3436366}.
\end{proof}

Let us denote by $\CLO(D_\infty)_{\xi}$ the closed subspace of $\CLO(D_\infty)$ of convex converging linear orders with root $\xi$. For brevity we denote $G_{\xi}=\Stab_{G_\infty}(\xi).$

%

\begin{lem}\label{minimality}Any $G_\infty$-orbit in $\CLO(D_\infty)$ is dense. Similarly, any $G_{\xi}$-orbit in $\CLO(D_\infty)_{\xi}$ is dense.
\end{lem}

\begin{proof}One has to show that for any pair $\prec_1,\prec_2\in\CLO(D_\infty)$ and any finite subset $F\subset\Br(D_\infty)$, there is $g\in G_\infty$ that induces an isomorphism from $(F,\prec_1)$ to $(gF,\prec_2)$ (i.e. for any $x,y\in F$, $x\prec_1 y\iff g(x)\prec_2g(y)$); and moreover if $\xi$ is the root of $\prec_1$ and $\prec_2$ then $g$ can be chosen in $G_{\xi}$.

For any finite set $F$ in a dendrite, the subdendrite $[F]$, that is the smallest subdendrite containing $F$, has finitely many branch points. So, up to add these branch points to $F$, we assume that $F$ is $c$-closed. We proceed by induction on the cardinality of $F$. If $F$ is reduced to a point then the result is immediate because $G_{\xi}$ acts transitively on branch points. Assume $F$ has $n\geq2$ points and we have the result for $n-1$. Let $m$ be the maximum of $F$ for $\prec_1$. The converging property of $\prec_1$ implies that $m$ is an end point of $[F]$ and thus $F'=F\setminus\{m\}$ is also a finite $c$-closed subset of $\Br(D_\infty)$ and thus by induction there is $g_1\in G_\infty$ that induces an isomorphism from  $(F',\prec_1)$ to $(g_1(F'),\prec_2)$. Moreover, if $\prec_1,\prec_2$ have root $\xi$, $g_1\in G_{\xi}$. It remains to put $m$ in the right position. 

\textbf{Claim:} Let $\prec$ be some convex converging linear order on $\Br(D_\infty)$, $x_1,x_2\in\Br(D_\infty)$ and $C_i\in\mathcal{C}_{x_i}$ such that $C_i$ does not contain the root of $\prec$. If $F$ is a finite $c$-closed subset such that $F\subset \overline{C_1}$ and $x_1\in F$ then there is a homeomorphism $h$ from $\overline{C_1}$ to $\overline{C_2}$ such that $h(x_1)=x_2$ and $h$ is increasing for $\prec$ on $F$.

\textbf{Proof of the claim:} Once again, we argue by induction and the case where $F=\{x_1\}$ is simply the fact there is a homeomorphism from $\overline{C_1}$ to $\overline{C_2}$ that maps $x_1$ to $x_2$. So assume $F$ has cardinality at least 2. Since $F'=F\setminus\{x_1\}$ is included in $C_1$, there is a minimal point $x'_1\in F'$ such that for any $y\in F$, $x’_1\in[x_1,y]$ and this point is in fact the minimum of $F\setminus\{x_1\}$. Let $C^1_1,\dots,C^k_1$ be the connected components of $C_1\setminus\{x_1’\}$ that meet $F'$ and let us denote $F_i=C^i_1\cap F'\cup\{x’_1\}$. We assume that these components are numbered increasingly ($C^i_1\prec^{x’}C^j_1\iff i<j$). Choose $x_2'\in\Br(C_2)$ and connected components $C^1_2,\dots,C^k_2\in\mathcal{C}_{x'_2}$ that do not contain $x_2$ and that are numbered increasingly. By the induction assumption, there is $h_i$ homeomorphism from $\overline{C^i_1}$ to $\overline{C^i_2}$ such that $h(x'_1)=x'_2$ and $h_i$ is increasing on $F_i$. Now choose any homeomorphism $h'$ from $\overline{C_1}\setminus\left(C^1_1,\dots,C^k_1\right)$ to $\overline{C_2}\setminus\left(C^1_2,\dots,C^k_2\right)$ such that $h'(x_1)=x_2$ and $h(x'_1)=x'_2$. Finally, we patch (thanks to Lemma~\ref{lem:patchwork}) $h',h_1,\dots,h_k$ to get a homeomorphism $h$ from $\overline{C_1}$ to $\overline{C_2}$. This homeomorphism is increasing on $F$ thanks to the convexity of $\prec$.

Let us come back to the end of the proof of the lemma. Let $x$ be the image of $m$ under the first point map on $F'=F\setminus\{m\}$. Since $F$ is $c$-closed, $x\in F$.   We denote by $C_1,\dots,C_k$ the (increasingly numbered for $\prec_2$) elements of $\mathcal{C}_{g(x)}$ that do not contain the root of $\prec_2$ and meet $g(F')$. Choose $C'_1,\dots,C'_{k+1}\in\mathcal{C}_{g(x)}\setminus\{C_1,\dots,C_k\}$ that do not contain the root of $\prec_2$ and are numbered increasingly. For each $i\leq k$, thanks to the claim, choose a homeomorphism $h_i$ from $\overline{C_i}$ to $\overline{C'_i}$ that fixes $g(x)$ and such that $h_i$ is increasing on $g(F')\cap C_i$. We also choose a homeomorphism $h_{k+1}$ from $C_{g(x)}(g(m))$ to $C_{k+1}$. Now, we define a homeomorphism $f$ of $D_\infty$ in the following way: $f|_{C_i}=h_i$ and  $f|_{C'_i}=h_i^{-1}$ (this is a legal definition because the $C_i$'s and $C_i'$'s are distinct) and $f$ is the identity elsewhere. This is a homeomorphism thanks to the patchwork lemma. Now $g=f\circ g_1$ is a homeomorphism and it is increasing thanks to the convexity of $\prec_2$. Observe that $g$ fixes $\xi$ if $g_1$ fixes $\xi$.
\end{proof}

\begin{lem}The action $\Stab_{G_\infty}(\prec_0)\action \Br(D_\infty)$ is oligomorphic.\end{lem}

\begin{proof} For a finite subset $F\subset \Br(D_\infty)$ of some fixed cardinality, there are finitely many possibilities for the order induced by $\prec_0$ on $F$. So it suffices to show that if $F, F’$ are two finite subsets of $\Br(D_\infty)$ such that the restriction of $\prec_0$ on $F$ and $F’$ are isomorphic then there is $g\in\Stab_{G_\infty}$ that induces this isomorphism. We have seen that  $(\Br(D_\infty),\leq,\wedge,\prec_0)$ is a Fraïssé limit (Proposition~\ref{fraise}) and $\Stab_{G_\infty}(\prec_0)$ is its automorphism group (Lemma~\ref{groupiso}). So such $g$ exists by the ultrahomogeneity of Fraïssé limits.
\end{proof}

\begin{thm}\label{univminflow} The universal minimal $G_\infty$-flow is $\CLO(D_\infty)$ and the universal minimal $G_{\xi}$-flow is  $\CLO(D_\infty)_{\xi}$.
\end{thm}

\begin{cor}The universal minimal $G_\infty$-flow  is metrizable and has a comeager orbit.
\end{cor}

Let us recall that any topological group $G$ has a left and right uniform structures. The right uniform structure $\mathcal{U}_r$ has a fundamental system of entourages given  by sets
$$\{(g,h)\in G\times G, gh^{-1}\in V\}$$
where $V$ is a symmetric neighborhood of the identity. For any closed subgroup $H$ this right uniform structure $\mathcal{U}_r$ yields  a uniform structure on the quotient space $G/H$ compatible with the quotient topology. A fundamental system of entourages is given by sets 
$$\{(gH,vgH),\ g\in G, v\in V\}$$ where $V$ is a symmetric neighborhood of the identity in $G$. We denote by $\widehat{G/H}$ the completion of $G/H$ with respect to this uniform structure. 

The closed subgroup $H$ is \emph{co-precompact} if $\widehat{G/H}$ is compact. This is equivalent to the following condition: for any neighborhood of the identity $V$, there is a finite subset $F\subset G$ such that $G=VFH$. If $G$ is an oligomorphic subgroup of $\mathcal{S}_\infty$ and $H$  a closed subgroup of $G$ then $H$ is co-precompact if and only if it itself oligomorphic. See \cite[\S 2]{MR3080786}.

Let $X$ be a $G$-flow and $x\in X$ be some $H$-fixed point. It is a standard fact the orbit map
$$\begin{array}{rcl}
G/H&\to& X\\
gH&\mapsto&gx\end{array}$$

is uniformly continuous and thus extends to a continuous map from $\widehat{G/H}$ to $X$. See \cite[Lemma 2.15 and \S 6.2]{MR2277969}.

\begin{proof}[Proof of Theorem~\ref{univminflow}]
For brevity, let us write $G=G_\infty$ (respectively $G=G_{\xi}$) and $H=\Stab(<_0)$. We identify $G/H$ with the $G$-orbit of $<_0$ in $\CLO(D_\infty)$. Since $\CLO(D_\infty)$ (respectively $\CLO(D_\infty)_{\xi}$) is a minimal $G$-flow and $H=\Stab(<_0)$ is an oligomorphic subgroup, thus co-precompact, we have a homeomorphism $\reallywidehat{G/H}\simeq \CLO(D_\infty)$ (respectively $\reallywidehat{G/H}\simeq \CLO(D_\infty)_{\xi}$). This follows from the fact that the identification $G/H$ with its orbit in $\CLO(D_\infty)$ is bi-uniformly continuous and $H$ is co-precompact, see  \cite[Corollary 1]{MR3080786}. Now, if $X$ is a minimal $G$-flow, since $H$ is extremely amenable, we have an orbit map $G/H\to X$, $gH\mapsto gx_0$ where $x_0$ is a $H$-fixed point. This maps is uniformly continuous and thus extends to $\reallywidehat{G/H}\to X$ and by minimality of $X$, it is surjective. So, $\reallywidehat{G/H}\simeq\CLO(D_\infty)$ (respectively $\CLO(D_\infty)_{\xi}$) is the universal minimal flow of $G$.
\end{proof}

\begin{rem}

In \cite{kwiatkowska2018universal}, written at the same time this work was done but this work was finalized much later, Aleksandra Kwiatkowska describes the universal minimal flow $M(G_S)$ of all Wa\.zewski groups $G_S$. So Theorem~\ref{univminflow} appears as a particular case of her results.  Her description of $M(G_S)$ allows to prove that $M(G_S)$ is metrizable if and only if $S$ is finite.\end{rem}



\section{Amenability and Furstenberg boundaries}
%
\subsection{Amenability}
Let $G$ be some topological group. Let us recall that a $G$-flow $X$ is \emph{strongly proximal} if the induced action on the $G$-flow of probability measures on $X$ is proximal. This is equivalent to the fact that for any probability measure $m$ on $X$, the adherence of the $G$-orbit of $m$ contains a Dirac mass \cite[Chapter III]{GlasnerLNM}. 

A topological group is \emph{amenable} if its universal Furstenberg boundary is a point, that is any strongly proximal minimal flow is trivial. Below, we recall a few conditions equivalent to amenability. Let $\ell^\infty(G)$ be the Banach space of all bounded functions on $G$. A function $f\in\ell^\infty(G)$ is \emph{right uniformly continuous} if the orbit map
$$\begin{array}{rcl}
G&\to&\ell^\infty(G)\\
g&\mapsto& R_g(f)
\end{array}$$

is continuous where $R_g(f)(h)=f(hg)$. Let us denote by $C_b^{ru}(G)$ the closed subspace of bounded right uniformly continuous functions on $G$ and let us observe that $G$ acts isometrically on $C_b^{ru}(G)$ by left translations $L_g(f)(h)=f(g^{-1}h)$. A \emph{mean} on $C_b^{ru}(G)$ is a linear functional $m$ such that 

\begin{enumerate}
\item $f\geq0\implies m(f)\geq0$,
\item $m(\mathbb{1}_G)=1$.
\end{enumerate}
Moreover, $m$ is said to be $G$-invariant if $m(L_g(f))=m(f)$ for all $f\in C_b^{ru}(G)$ and all $g\in G$.

\begin{thm}\label{amenequiv}If $G$ is a topological group, the following conditions are equivalent.
\begin{enumerate}
\item $G$ is amenable,
\item any $G$-flow has an invariant probability measure, 
\item any affine $G$-flow has a fixed point,
\item any strongly proximal $G$-flow has a fixed point,
\item there is a $G$-invariant mean on $C_b^{ru}(G)$.
\end{enumerate}
\end{thm}

A proof of this theorem can be found in \cite[Theorem III.3.1]{GlasnerLNM}. 
%
%

\begin{lem}\label{amen} A topological group $G$ is amenable if and only if there is an invariant probability measure on its universal minimal flow $M(G)$.

\end{lem}

\begin{proof} The condition is clearly necessary. Let us show that  it is sufficient. Let $X$ be a $G$-flow. Let us choose a minimal subflow $X_0$. By the universal property of $M(G)$, there is a continuous surjective $G$-map $M(G)\to X_0$. The image of an invariant probability measure on $M(G)$ is an invariant probability measure on $X_0$ and thus one gets an invariant probability measure on $X$ and thus $G$ is amenable.
\end{proof}

Let us fix an end $\xi\in D_\infty$ and denote by $\CLO(D_\infty)_\xi$, the subset of $\CLO(D_\infty)$ of orders $\prec$ with root $\xi$. This condition is equivalent to
$$\forall y\in\Br(D_\infty),\ y\in]x,\xi]\implies y\prec x.$$

For a branch point $b$, let us recall that $\mathcal{C}_{b,\xi}$ is the space $\mathcal{C}_b\setminus C_b(\xi)$. For any countable set $X$, we denote by $\LO(X)$ the set of linear orders on $X$ with its usual topology as a closed subspace of $\{0,1\}^{X^2\setminus\Delta}$.
\begin{lem}\label{identification} The subset $\CLO(D_\infty)_\xi$ is closed in $\CLO(D_\infty)$ and homeomorphic to the product space $\Pi_{b\in\Br(D_\infty)}\LO(\mathcal{C}_{b,\xi})$.\end{lem}

\begin{proof}The condition of convergence to $\xi$ is given by a collection of closed conditions. Thus $\CLO(D_\infty)_\xi$ is closed in $\CLO(D_\infty)$.

We have seen in Lemma \ref{obsorder} that any convex converging linear order $\prec$ induces a linear order $\prec^b$ on branches around $b$ that do not contain the root. So we get a continuous map 

$$\begin{array}{rlc}
\CLO(D_\infty)_\xi&\to&\prod_{b\in\Br(D_\infty)}\LO(\mathcal{C}_{b,\xi})\\
\prec&\mapsto&(\prec^b).
\end{array}$$

Conversely, from $(\prec^b)$, we can construct an order $\prec$ by defining $a\prec b$ if and only if $a=c(a,b,\xi)$ or $a\prec^c b$ where $c=c(a,b,\xi)$. One can check, as in Lemma~\ref{cclo},  that this definition yields an element in $\CLO(D_\infty)_\xi$ and this operation is the inverse of the map above.
\end{proof}

In the remaining of this section, we denote by $G$ the group $G_\infty$ and by $G_\xi$ the stabilizer of the end point $\xi$.

\begin{prop}\label{invmeas}There is a $G_\xi$-invariant measure on $\CLO(D_\infty)_\xi$.
\end{prop}

If $f$ is a bijection between two countable sets $X$ and $Y$, it induces a bijection $f_*$  between linear orders on $X$ and on $Y$. If $\prec\in\LO(X)$, $f_*\!\prec\in\LO(Y)$ is defined by $y(f_*\!\prec)y'\iff f^{-1}(y)\prec f^{-1}(y')$.

Lemma~\ref{identification} gives an identification between $\CLO(D_\infty)$ and $\Pi_{b\in\Br(D_\infty)}\LO(\mathcal{C}_{b,\xi})$. Let us describe how $G_\xi$ acts on the product via this identification. Any $g\in G_\xi$ induces a bijection $\sigma(g,b)\colon \mathcal{C}_{b,\xi}\to\mathcal{C}_{gb,\xi}$.
Now, if $\prec\in\CLO(D_\infty)$ corresponds to $(\prec^b)_{b\in\Br(D_\infty)}$ then $g_*\!\prec$ corresponds to $\left(\sigma(g,g^{-1}b)_*\!\prec^{g^{-1}b}\right)_{b\in\Br(D_\infty)}$.

\begin{proof}Let us choose $b_0\in\Br(D_\infty)$ and for each $b\in\Br(D_\infty)$, fix some bijection $f_b\colon \mathcal{C}_{b_0,\xi}\to\mathcal{C}_{b,\xi}$. For example, this bijection can be induced by an element $g\in G_\xi$ such that $g(b_0)=b$. 

Since $\mathcal{S}_\infty$ is amenable, there is an invariant probability $\mu_0$ on $\LO\left(\mathcal{C}_{b_0,\xi}\right)$ under all bijections of $\mathcal{C}_{b_0,\xi}$. Let us denote $\mu_b=(f_b)_*\mu_0$ that is a probability measure on $\LO\left(\mathcal{C}_{b,\xi}\right)$ and finally set $\mu$ to be the product measure of all $\mu_b$ on $\prod_{b\in\Br(D_\infty)}\LO(\mathcal{C}_{b,\xi})\simeq \CLO(D_\infty)$. We aim to prove that for any $g\in G$, $g^*\mu=\mu$. It suffices to prove this equality on cylinders. So, choose distincts $b_1,\dots,b_n\in \Br(D_\infty)$ and mesurable sets $A_i\subset \LO(\mathcal{C}_{b_i,\xi})$ and set $A$ to be the cylinder $$A=\prod_{i=1}^n A_i\times \prod_{b\neq b_i}\LO\left(\mathcal{C}_{b,\xi}\right).$$ 
One has $$g_*(A)=\prod_{i=1}^n \sigma(g,b_i)_*A_i\times\prod_{b\neq gb_i}\LO\left(\mathcal{C}_{b,\xi}\right)$$
and thus $g^*\mu(A)=\mu(g_*(A))=\prod_{i=1}^n\mu_{gb_i}(\sigma(g,b_i)_*(A_i))$. One can compute

\begin{align*}
\mu_{gb_i}\left(\sigma(g,b_i)_*(A_i)\right)&=\mu_0\left(f_{gb_i}^{-1}\left(\sigma(g,b_i)_*(A_i)\right)\right)\\
&=\mu_0\left((f_{gb_i}^{-1}\circ\sigma(g,b_i)\circ f_{b_i})_*((f_{b_i}^{-1})_*(A_i))\right).\end{align*}
Since $f_{gb_i}^{-1}\circ\sigma(g,b_i)\circ f_{b_i}$ is a bijection of $\mathcal{C}_{b_0,\xi}$ and $\mu_0$ is invariant under $\Sym\left(\mathcal{C}_{b_0,\xi}\right)$, 
\begin{align*}
\mu_{gb_i}\left(\sigma(g,b_i)_*(A_i)\right)&=\mu_0\left((f_{b_i}^{-1})_*(A_i)\right)\\
&=\mu_{b_i}(A_i)
\end{align*}
and thus $\mu\left(g_*(A)\right)=\mu(A)$.
\end{proof}
As a consequence of Theorem~\ref{univminflow}, Proposition \ref{invmeas} and Lemma \ref{amen}, one has the following.
\begin{thm}The topological group $G_\xi$ is amenable.
\end{thm}

Finally, Theorem~\ref{ufb} is obtained as the following corollary.

\begin{cor}For any point $x\in D_\infty$, the stabilizer $G_x$ of $x$ in $G$ is amenable.
\end{cor}

\begin{proof} Thanks to \cite[Lemma 7.1]{DM_dendritesII}, for $x\in D_\infty\setminus\Ends(D_\infty)$, $G_x$ splits homeomorphically as $\left(\prod_{i\in\NN}G_{\xi_i}\right)\rtimes \mathcal{S}_\infty$ if $x\in\Br(D_\infty)$ and splits as $G_\xi^2\rtimes\ZZ/2\ZZ$ if $x$ is a regular point. As it is well known, a product of amenable groups is amenable and an extension of an amenable group by another amenable group is amenable as well. So $G_x$ is amenable.
\end{proof}

\begin{rem} For any finite subset $F\subset D_\infty$, the stabilizer and the pointwise stabilizer of $F$ are amenable groups.\end{rem}

\subsection{Universal Furstenberg boundary}
Let $\varphi\colon G/G_\xi\to D_\infty$ be the continuous orbit map $gG_\xi\mapsto g\xi$. Since the $G$-orbit of $\xi$, that is $\Ends(D_\infty)$, is a $G_\delta$ in $D_\infty$, Effros theorem (\cite[Theorem 7.12]{hjorth2000classification}) implies that $\varphi$ is a homeomorphism on its image.

This map $\varphi$ is uniformly continuous for the uniform structure coming from the right uniform structure on $G$, thus it extends uniformly continuously to a $G$-equivariant surjective map $\overline{\varphi}\colon\widehat{G/G_\xi}\to D_\infty$. 
A fundamental system of entourages for the uniform structure on $G/G_\xi$ coming from the right uniform structure on $G$, is given by sets

$$U_V=\{(gG_\xi,vgG_\xi),\ v\in V, g\in G\}$$
where $V$ is symmetric neighborhood of the identity in $G$. This uniform structure is metrizable (see the introduction of \cite{melleray2015polish} for example). Let us push forward this uniform structure on $\Ends(D_\infty)$ via $\varphi$. So, a fundamental system of entourages for this uniform structure is given by sets

$$U_V=\{(\zeta,\eta),\ \exists v\in V,\ \eta=v\zeta\}$$

where $V$ is symmetric neighborhood of the identity in $G$. Let us denote by $\overline{\Ends(D_\infty)}$ the completion of $\Ends(D_\infty)$ by this uniform structure. So, this space $\overline{\Ends(D_\infty)}$ is isomorphic to $\widehat{G/G_\xi}$ as uniform $G$-space but we introduce it because we think this is more convenient to speak about Cauchy sequences of ends points instead of Cauchy sequences of $G_\xi$-cosets.
For $b\in\Br(D_\infty)$ and $\eta\in\Ends(D_\infty)$, we denote by $\overline{C_b(\eta)}$ the adherence of  $C_b(\eta)$ in $\overline{\Ends{D_\infty}}$.

\begin{lem}\label{basis}Let $\xi\in \Ends(D_\infty)$. Let $(b_n)$ be a sequence of branch points in $D_\infty$ converging to $\xi$. The collection $\left\{\overline{C_{b_n}(\xi)}\right\}$  is a basis of neighborhoods of $\xi$ in $\overline{\Ends(D_\infty)}$.
\end{lem}

\begin{proof}Let us first show that $\overline{C_b(\xi)}$ is a neighborhood of $\xi$ in $\overline{\Ends(D_\infty)}$. Let us choose a branch point $b'\in]b,\xi[$. Now for any $g\in G_\infty$ fixing $b$ and $b’$, $g\xi\in C_b(\xi)$. In particular, $U_{V_{\{b,b'\}}}(\xi)=\{\eta\in\Ends(D_\infty),\ \exists g\in V_{\{b,b'\}},\ \eta=g\xi\}\subset C_b(\xi)$ and thus $\overline{U_{V_{\{b,b'\}}}(\xi)}\subset \overline{C_b(\xi)}$ which shows that $\overline{C_b(\xi)}$ is a neighborhood of $\xi$ in $\overline{\Ends(D_\infty)}$.

Let $F$ be some finite subset of $\Br(D_\infty)$. Let $F'$ be the $c$-closure of $F$. The intersection $\cap_{b\in F'}C_b(\xi)\subset U_{V_F}(\xi)=\{\eta\in\Ends(D_\infty),\ \exists g\in V_F,\ \eta=g\xi\}$ because if $\eta\in\cap_{b\in F’}C_b(\xi)$ then $\xi$ and $\eta$ lie in the same connected component of $D_\infty\setminus F'$. This component has at most two points in its boundary. The labelled graphs $\langle F'\cup\{\xi\}\rangle$ and $\langle F'\cup\{\eta\}\rangle$ are isomorphic and thus, one can find $g\in V_{F'}$ such that $g\xi=\eta$ by Proposition \ref{prop:extension}.
So we have $\cap_{b\in F'}C_b(\xi)\subset U_{V_F}(\xi)$ and $\cap_{b\in F'}\overline{C_b(\xi)}\subset \overline{U_{V_F}(\xi)}$. Choose $n$ large enough such that $b_n\in \cap_{b\in F'} C_b(\xi)$. One has $C_{b_n}(\xi)\subset \cap_{b\in F'} C_b(\xi)$ and the same holds for the adherences. This shows that the collection $\left\{\overline{C_{b_n}(\xi)}\right\}$ is a basis of neighborhoods of $\xi$.
\end{proof}

We can now prove Theorem~\ref{ufba}, that is $\widehat{G/G_\xi}$ is the universal Furstenberg boundary of $G$.
\begin{proof}[Proof of Theorem~\ref{ufba}]

Since $G_\xi$ is oligomorphic, $\widehat{G/G_\xi}$ is compact. Let $H$ be the stabilizer of $\prec_0$ from Lemma \ref{cclo}. Since $H$ fixes $\xi$, the uniformly continuous map $G/H\to G/G_\xi$ extends continuously to an equivariant surjective map $\widehat{G/H}\to\widehat{G/G_\xi}$. The minimality of $\widehat{G/H}$ implies the one of $\widehat{G/G_\xi}$.

We have seen that $G_\xi$ is amenable. Let $X$ be a minimal strongly proximal $G$-flow. By amenability, there is a $G_\xi$-fixed point $x$. The orbit map $gG_\xi\mapsto gx$ extends continuously to a $G$-map $\widehat{G/G_\xi}\to X$ an by minimality this map is surjective. 

It remains to show that the action of $G$ on $\widehat{G/G_\xi}$ is strongly proximal. We follow the strategy that was used in the proof that the action of $G_\infty$ on $D_\infty$ is strongly proximal \cite[Theorem 10.1]{DM_dendrites}.
Let $m$ be a Borel probability measure on $\overline{\Br(D_\infty)}$. Since $\Ends(D_\infty)$ is uncountable, there is $\eta\in\Ends(D_\infty)$ such that $m(\{\eta\})=0$. Let $\eta’$ be another end point, $(b_n), (b'_n)$ be sequences of branch points in $[\eta,\eta']$ converging respectively to $\eta$ and $\eta'$. Thanks to Lemma~\ref{basis}, $m\left(\overline{C_{b_n}(\eta)}\right)\to 0$ and thus $m\left(\overline{\Ends(D_\infty)}\setminus\overline{C_{b_n}(\eta)}\right)\to 1$. Let $g_n\in G$ fixing $\eta,\eta'$ and such that $g_nb_n=b'_n$. For any $b\in\B$, one can find $n$ large enough such that $b'_n\in C_b(\eta')$ and thus $\overline{\Ends(D_\infty)}\setminus\overline{C_{b'_n}(\eta)}\subset \overline{C_b(\eta')}$. This shows that $(g_n)_\ast m(\overline{C_b(\eta')})\to 1$ and thus $(g_n)_\ast m\to \delta_{\eta'}$.
\end{proof}

\begin{rem}
The universal Furstenberg boundary of $G_\infty$ can also be recovered from 	\cite[Theorem 7.5]{2018arXiv181200392Z} and Theorem~\ref{ufb}.\end{rem}

\begin{prop}\label{nothomeo} The map $\overline{\varphi}\colon\widehat{G/G_\xi}\to D_\infty$ is not a homeomorphism.\end{prop}

\begin{proof}We continue to identify $\widehat{G/G_\xi}$ with $\overline{\Ends(D_\infty)}$. Since the spaces $\widehat{G/G_\xi}$ and  $D_\infty$ are compact, they have a unique uniform structure and thus it suffices to show there is a sequence $(\xi_n)$ of end points which is Cauchy in $D_\infty$ but not in $\overline{\Ends(D_\infty)}$. Let $b\in \Br(D_\infty)$ and $C_1\neq C_2\in \mathcal{C}_b$. Choose $(\xi_{2n})$ sequence of end points of $C_1$ converging to $b$ in $D_\infty$ and similarly, choose $(\xi_{2n+1})$ sequence of $C_2$ converging to $b$ in $D_\infty$. The sequence $(\xi_n)$ converges in $D_\infty$ and thus is Cauchy but if $b'\in C_1$ and $F=\{b,b'\}$ then there is no $g\in V_F$ such that $gC_2=C_1$ and thus for any $n,m\in\NN$, $(\xi_{2m},\xi_{2n+1})\notin U_{V_F}$ and thus $(\xi_n)$ is not Cauchy in $\overline{\Ends(D_\infty)}$. 
\end{proof}

\subsection{Another description of the universal Furstenberg boundary}\label{isomorphic} Let us finish this paper with another description of the universal Furstenberg boundary of $G$. This will be the compact space $K$ below. 

For each $b\in\Br(D_\infty)$, let us consider $\mathcal{C}_b$, the space of connected component around $b$, with the discrete topology. Let $\overline{\mathcal{C}_b}$ be its Alexandrov compactification and let us denote by $C_b^\infty$  the added point. The product $\prod_{b\in\Br(D_\infty)}\overline{\mathcal{C}_b}$ is a metrizable totally disconnected compact space. The group $G$ acts continuously on this product space in the following way

$$g(C_b)_{b}=(g(C_{g^{-1}b}))_b$$

where we use the convention $g(C_b^\infty)=C_{gb}^\infty$ for any $g\in G$ and $b\in \Br(D_\infty)$. Let us define

$$K=\left\{(C_b)\in \prod_{b\in\Br(D_\infty)}\overline{\mathcal{C}_b},\ \forall b,b’,\ b’\notin C_b \implies  C_{b’}=C_{b’}(b)\right\}.$$

where $C_b(x)$ is the element of $\mathcal{C}_b$ that contains $x$ for $x\neq b$. For $C_b^\infty$, we use the convention that for any $b’\in\Br(D_\infty)$, $b’\notin C_b^\infty$.
For $x\in D_\infty\setminus\Br(D_\infty)$, let $C(x)\in \prod_{b\in\Br(D_\infty)}\overline{\mathcal{C}_b}$ be $(C_b(x))_b$. For each $b\in\Br(D_\infty)$, let us enumerate $\mathcal{C}_{b}=\{C_b^n\}_{n\in\NN}$.  For $n\in\overline{\NN}$ and $b_0\in\Br(D_\infty)$ we define $C^n(b_0)$ to be $(C_b)$ where $C_b=C_b(b_0)$ for $b\neq b_0$ and $C_{b_0}=C_{b_0}^n$. 

\begin{lem}\label{lem:orbits} The space $K$ is a closed $G$-invariant subset. For any $C\in K$, there is $x\in D_\infty\setminus\Br(D_\infty)$ such that $C=C(x)$ or there is $(b,n)\in \Br(D_\infty)\times\overline{\NN}$ such that $C=C^n(b)$.

\end{lem}

\begin{proof} The conditions that define $K$ are $G$-invariant and closed for the product topology. 
We claim that if $C\in K$ then there is at most one $b$ such that $C_b=C^\infty_b$. Assume there are $b\neq b’$ that satisfy this condition. Choose $b\mydprime\in]b,b’[$ then $C_{b\mydprime}$ should be $C_{b\mydprime}(b)$ and $C_{b\mydprime}(b’)$ but these components are distinct.

Let $C\in K$. For any $b,b’$ such that $C_b\neq C_b^\infty$ and $C_{b’}\neq C_{b’}^\infty$ then $C_b\cap C_{b’}\neq\emptyset$ because either one is included in the other or $]b,b’[\subset C_b\cap C_{b’}$. By the Helly’s property \cite[Lemma 2.1]{DM_dendrites}, the intersection $\bigcap \overline{C_b}$ over all $b$’s such that $C_b\neq C_b^\infty$ is non-empty and convex. If  $x\neq y$ lie in this intersection then for $b\in]x,y[\cap\Br(D_\infty))$, $C_b=C_b(x)$ and $C_b=C_b(y)$ which is impossible. Thus this intersection is reduced to one point. 

We conclude this lemma by observing that if $x$ is this intersection point then for any branch point $b\neq x$ then $C_b=C_b(x)$.
\end{proof}

\begin{lem} The map $\Ends(D_\infty)\to K$ given by $\xi\mapsto C(\xi)$ is $G$-equivariant and injective. Moreover the image is dense in $K$.
\end{lem}

\begin{proof} The $G$-equivariance is the following straightforward computation:

$$gC(\xi)=(g(C_{g^{-1}b}(\xi))_b=(C_{b}(g\xi))_b=C(g\xi).$$

It is injective because if $\xi,\eta\in\Ends(D_\infty)$ are distinct then for any $b\in]\xi,\eta[\cap\Br(D_\infty)$, $C_b(\xi)\neq C_b(\eta)$. To prove density it suffices to show that for any $C=(C_b)\in K$ there is $(\xi_n)$ sequence of  $\Ends(D_\infty)$ such that for any $b\in \Br(D_\infty)$, $C_b(\xi_n)\to C_b$.

If $C=C(x)$ for $x$ non-branch point then for any sequence $(\xi_n)$ converging to $x$ in $D_\infty$ will be suitable because for any $b$, $C_b(x)$ is open and contains $x$ thus $C_b(\xi_n)=C_b(x)$ for $n$ large enough. If $C=C^k(b)$ with $k$ finite then any sequence $(\xi_n)$ of end points in $C_b^k$ converging to $x$ in $D_\infty$ will be suitable for the same argument. Finally, if $C=C^\infty(b)$ then a sequence $(\xi_n)$ such that $\xi_n\in C^n(b)$ for any $n\in\NN$ will be suitable. Actually for any $b’\neq b$,  $C_{b’}(\xi_n)=C_{b’}(b)$ for all $n$ except at most one and $C_b(\xi)\to C_b^\infty$ because $\xi_n$ eventually leaves any finite union of elements of $\mathcal{C}_b$.
\end{proof}

\begin{prop} The spaces $\overline{\Ends{D_\infty}}$ and $K$ are isomorphic as $G$-flows. Moreover, there are countably many $G$-orbits.
\end{prop}

\begin{proof} The spaces $\overline{\Ends{D_\infty}}$ and $K$ are compact and metrizable so it suffices to prove that for any sequence for end points $(\xi_n)$, $(\xi_n)$ is Cauchy in $\overline{\Ends{D_\infty}}$ if and only if $(C(\xi_n))$ is Cauchy in $K$. That is,  $(\xi_n)$ converges in $\overline{\Ends{D_\infty}}$ if and only if $(C(\xi_n))$ converges in $K$. This will show that the two spaces are homeomorphic and the existence of a dense $G$-orbit will imply that the homeomorphism is $G$-equivariant.

Let $(\xi_n)$ be a convergent sequence in $\overline{\Ends{D_\infty}}$. This means that for any finite set $F$ of branch points, there is $N\in \NN$ such that for any $n,m\geq N$, there is $g\in G$ fixing pointwise $F$ and such that $g\xi_n=\xi_m$.

Let $b\in \Br(D_\infty)$. For any $k\in\NN$, choose a branch point $b_k\in C_b^{k}$. Since for any element $g\in G$ fixing $b$ and $b_k$, $C^k_b$ is $g$-invariant, one has that either eventually $C_b(\xi)=C_b^k$ or eventually $C_b(\xi)\neq C^k_b$. Thus $(C_b(\xi))$ is convergent in $\overline{\mathcal{C}_b}$ and $C(\xi_n)$ is convergent in $K$.

Conversely assume that $(\xi_n)$ is  a sequence of end points such that $C(\xi_n)$ is convergent in $K$. Let $F$ be some finite set of branch points. Up to enlarge $F$, we may assume that $F$ is $c$-closed. 

First let us assume there is $b\in F$ such that $C_b(\xi_n)\to C_b^\infty$. 
Let $N$ such that for $n\geq N$, $C_b(\xi_n)\neq C_b(b’)$ for all $b’\in F\setminus\{b\}$. For $n,m\geq N$, choose $g\in G$, switching $C_b(\xi_n)$ and $C_b(\xi_m)$ such that $g\xi_n=\xi_m$ and such that $g$ is the identity on $D_\infty\setminus (C_b(\xi_n)\cup C_b(\xi_m))$. In particular, g fixes pointwise $F$ and thus $\xi_n$ is convergent in $\overline{\Ends{D_\infty}}$. 

Now, assume there is no $b\in F$ such that $C_b(\xi_n)\to C_b^\infty$. This means that for any $b\in F$, $C_b(\xi_n)$ is eventually equal to some $C_b$. The intersection  $\cap_{b\in F} C_b$ is one of the connected component of $D_\infty\setminus F$. Since $F$ is $c$-closed, there are at most two elements of $F$  in its boundary. If there is only one then this intersection $\cap_{b\in F} C_b$ is some $C_b$, which do not contain any $b’\in F$ and for $n,m$ large enough, $\xi_n,\xi_m\in C_b$ and thus one can find $g\in G_\infty$ fixing pointwise $D_\infty\setminus C_b$ such that $g\xi_n=\xi_m$. If there are two points $b_1,b_2$ in the boundary then $\cap_b C_b=C_{b_1}\cap C_{b_2}=D(b_1,b_2)$ does not contain any $b\in B$. So, for $n,m$ large enough, $\xi_n,\xi_m\in D(b_1,b_2)$ and thus one can find $g\in G_\infty$ fixing pointwise $D_\infty\setminus D(b_1,b_2)$ such that $g\xi_n=\xi_m$. In both cases $(\xi_n)$ is convergent in $\overline{\Ends{D_\infty}}$.

The statement about the number of orbits follows from Lemma~\ref{lem:orbits}. Any $C\in K$ is of the form $C=C(x)$ for $x$ non-branch point or $C=C^n(b)$ for some $n\in\overline{N}$ and some branch point $b$. In the first case, this gives two orbits, depending wether $x$ is a regular or an end point. In the second case, this gives countably many orbits, one for each $n\in\overline{N}$.
\end{proof}

\begin{rem} Since $K$ is totally disconnected and $D_\infty$ is connected, this gives another proof of Proposition~\ref{nothomeo}.
\end{rem}
%

\bibliographystyle{halpha}
\bibliography{dendrite}
\end{document}